\documentclass[12pt]{article}

\usepackage {graphicx}
\usepackage{amssymb}
\usepackage{amsmath}
\usepackage{amsbsy}
\usepackage{setspace}
\usepackage{natbib}
\usepackage{bm}
\usepackage{textcomp}
\usepackage{color}

\usepackage{amsthm}
\usepackage{enumitem}
\usepackage{longtable,threeparttablex,booktabs}
\usepackage{hyperref}
\usepackage{mathtools}
\usepackage{lscape}
\usepackage{subfigure}
\usepackage{multirow}
\usepackage{lineno}

\newtheorem{theorem}{Theorem}
\newtheorem{corollary}{Corollary}
\newtheorem{proposition}{Proposition}
\newtheorem{lemma}{Lemma}
\newtheorem{example}{Example}

\newtheorem{definition}{Definition}

\def\text#1{\mbox{\rm #1}}
\def\overset#1#2{\stackrel{#1}{#2} }

\def\underwiggle 1{
	\ifmmode\setbox\TempBox=\hbox{$ 1$}\else\setbox\TempBox=\hbox{
		1}\fi \setbox\TempBoxA=\hbox to \wd\TempBox{\hss\char'176\hss}
	\rlap{\copy\TempBox}\smash{\lower9pt\hbox{\copy\TempBoxA}} }

\newcommand{\E}{\mathrm E}
\newcommand{\df}{\,\mathrm d}

\newcommand{\var} {\mbox{var}}

\newcommand{\beq}{\begin{equation}}
	\newcommand{\eeq}{\end{equation}}
\newcommand{\beas}{\begin{eqnarray*}}
	\newcommand{\eeas}{\end{eqnarray*}}
\newcommand{\bea}{\begin{eqnarray}}
	\newcommand{\eea}{\end{eqnarray}}
\newcommand{\bei}{\begin{itemize}}
	\newcommand{\eei}{\end{itemize}}
\newcommand{\ben}{\begin{enumerate}}
	\newcommand{\een}{\end{enumerate}}
\newcommand{\bet}{\begin{theorem}}
	\newcommand{\eet}{\end{theorem}}
\newcommand{\bel}{\begin{lemma}}
	\newcommand{\eel}{\end{lemma}}
\newcommand{\bep}{\begin{proposition}}
	\newcommand{\eep}{\end{proposition}}
\newcommand{\bed}{\begin{definition}}
	\newcommand{\eed}{\end{definition}}
\newcommand{\bec}{\begin{corollary}}
	\newcommand{\eec}{\end{corollary}}
\newcommand{\bex}{\begin{example}}
	\newcommand{\eex}{\end{example}}

\begin{document}
	\title{\textcolor{black}{Regression analysis of longitudinal data with mixed synchronous and asynchronous longitudinal covariates}}
	
	\author{Zhuowei Sun$^1$, Hongyuan Cao$^2$, Li Chen$^3$ and Jason P. Fine$^4$}
	
	\footnotetext[1]{~ School of Mathematics, Jilin University, Changchun 130012, China.}
	\footnotetext[2]{~ Department of Statistics, Florida State University, Tallahassee, FL 32306, U.S.A.}
	\footnotetext[3]{~ Novartis Pharmaceuticals Corp. 1 Health Plaza, East Hanover, NJ, 07936, U.S.A.}
\footnotetext[4]{~ Department of Biostatistics and Department of Statistics and Operations Research, University of North Carolina at Chapel Hill, Chapel Hill, NC 27514, U.S.A.}

\date{}
\maketitle

\begin{abstract}
	In linear models, omitting a covariate that is orthogonal to covariates in the model does not result in biased coefficient estimation. This in general does not hold for longitudinal data, where additional assumptions are needed to get unbiased coefficient estimation in addition to the orthogonality between omitted longitudinal covariates and longitudinal covariates in the model. We propose methods to mitigate the omitted variable bias under weaker assumptions.
	A \textcolor{black}{two-step estimation procedure} is proposed for inference about the asynchronous longitudinal covariates, when such covariates are observed. For mixed synchronous and asynchronous longitudinal covariates, we get parametric rate of convergence for the coefficient estimation of the synchronous longitudinal covariates by the two-step method. Extensive simulation studies provide numerical support for the theoretical findings. We illustrate the performance of our method on dataset from the Alzheimers Disease Neuroimaging Initiative study.
\end{abstract}

\thispagestyle{empty}

\section{Introduction}\label{intro.sec}
\setcounter{page}{1}
In linear models, the Frisch-Waugh-Lovell (FWL) theorem states the equivalence of the coefficients from the full and partial regression. Specifically, using projection matrices to make the explanatory variables orthogonal to each other will lead to the same results as running the regression with all non-orthogonal explanators included \citep{ding2021,frisch1933,lovell1963,lovell2008}. In particular, if we omit a variable orthogonal to variables in the model, we get unbiased regression coefficient estimation. Recently, this idea is used to obtain causal treatment effect estimation for longitudinal data \citep{bates2022}. Do similar results hold in longitudinal studies with time-dependent covariates? How do we get unbiased regression coefficient estimation with omitted longitudinal covariate? Furthermore, what if the omitted longitudinal covariate is asynchronous with the longitudinal response and other longitudinal covariates in the model? 

Asynchronous longitudinal data refer to the misalignment of measurement times on two longitudinal processes within an individual. Typical examples arise in the analysis of electronic health records (EHR) data, where patients' health information are collected from multiple sources. EHR may include data on individual's demographics, medication and allergies, immunization status, laboratory test results and billing information, among others. Due to the retrospective nature of EHR, the measurement times are collected at each  clinical encounter, which can be irregular and sparse across patients and asynchronous within patients. Another example comes from the Alzheimer's Disease Neuroimaging Initiative study (ADNI), where cognitive decline metrics, such as Mini-Mental State Examination (MMSE) score, are misaligned with medical imaging measurement, such as log hazard of fractional anisotropy (FA), which reflects fiber density, axonal diameter, and myelination in white matter, within an individual. Typical measurement times of a few patients in this dataset are plotted in Figure \ref{fig0}. We can see that each patient has different number of measurements of MMSE and FA. Furthermore, the measurement times of MMSE and FA are mismatched. \begin{figure}[!ht]
	\includegraphics[width = 4.5in]{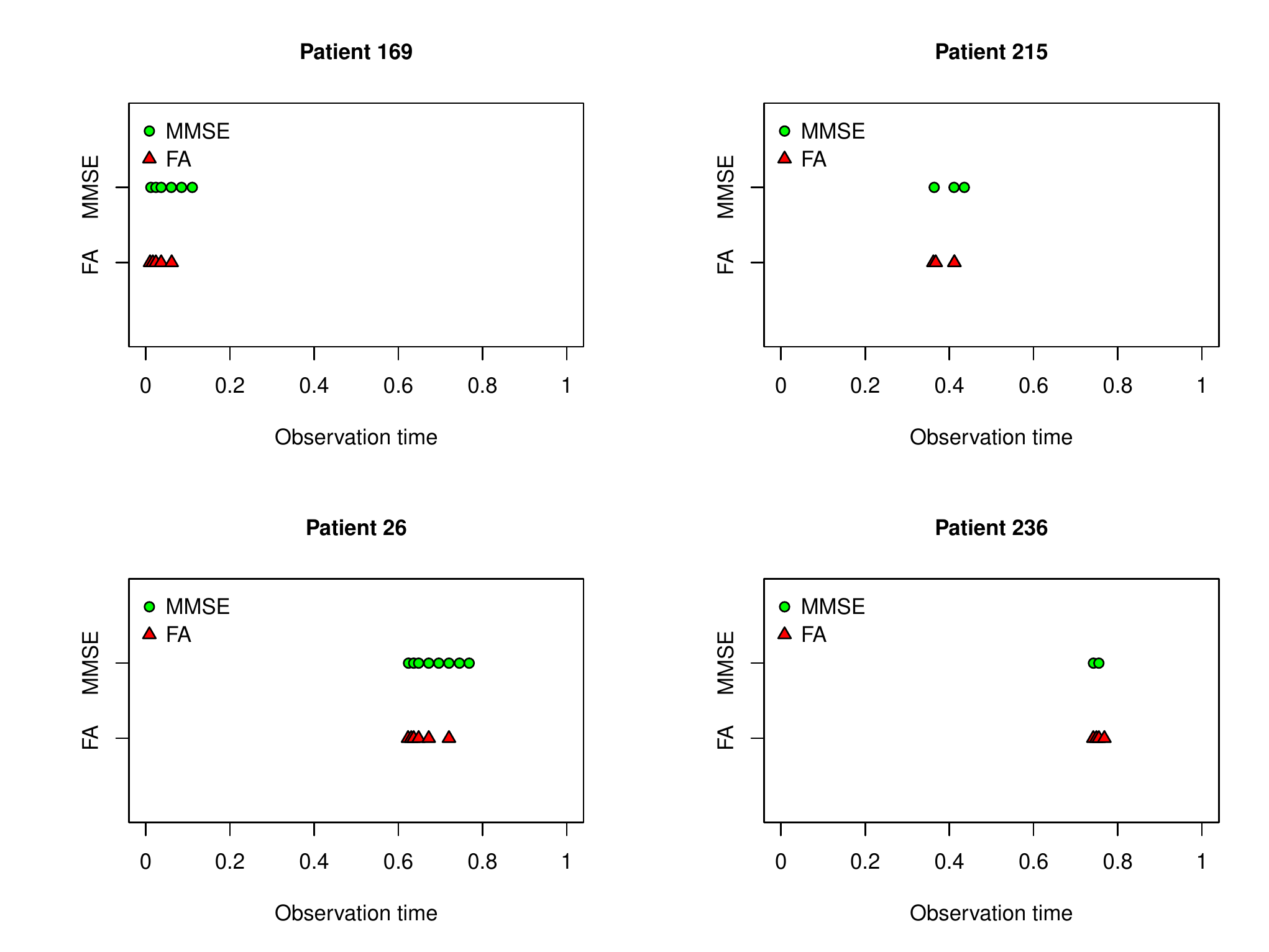}
	\centering \caption{\label{fig0}Examples of individual observations}
\end{figure}

For asynchronous longitudinal data, \cite{xiong10} employed a binning approach to synchronize covariates and response measurements to use exiting methods for classic longitudinal data analysis. \cite{senturk13} explicitly addressed the asynchronous structure for generalized varying-coefficient models with one covariate yet did not provide any theoretical justification. \cite{cao15} proposed a non-parametric weighting approach for generalized linear models with asynchronous longitudinal data and rigorously established inferential strategies. This was extended to a more general set up in \cite{cao16} and partial linear model in \cite{chen17}. Recently, \cite{lizhu2022} studied temporally asynchronous functional imaging data and \cite{sun2021} examined informative measurement times for asynchronous longitudinal data. These approaches assume that all asynchronous longitudinal covariates have the same measurement times, which are asynchronous with longitudinal response. The problem of mixed synchronous and asynchronous longitudinal covariates have not been addressed before. 

In this paper, we propose statistical methods for analyzing mixed synchronous and asynchronous longitudinal covariates. The longitudinal covariates have two sets, one set is measured synchronously with the longitudinal response and another set is measured asynchronously with the longitudinal response. Suppose we are interested in inference on the synchronous longitudinal covariates and treat the asynchronous longitudinal covariates as nuisance. Unlike classic linear models, unbiased regression coefficient estimation of the synchronous longitudinal covariates usually cannot be obtained when omitting the asynchronous longitudinal covariates even if they are uncorrelated unless the omitted asynchronous longitudinal covariates have constant expectation over time. Ignoring this fact and only fitting synchronous longitudinal covariates with the longitudinal response may incur omitted variable bias. 

To mitigate such bias, we can fit synchronous and asynchronous longitudinal covariates simultaneously like that in \cite{cao15}. For synchronous longitudinal covariates, this one-step method implements an unnecessary smoothing, which slows down rate of convergence of the regression coefficient. To improve efficiency, we propose a two-step method. In the first step, we either fit a partial linear model of synchronous longitudinal covariate with respect to the longitudinal response or a linear model with centered synchronous longitudinal covariate and centered longitudinal response, omitting the asynchronous longitudinal covariates. Intuitively, we either absorb the omitted longitudinal covariates through the non-parametric intercept in the partial linear model or get rid of them through centering. We show that parametric rate of convergence can be obtained for the regression coefficient estimation of synchronous longitudinal covariates. In the second step, residuals from the first step are fitted with the asynchronous longitudinal covariates by kernel weighting. It is established that the resulting estimator is consistent, asymptotically normal, and has the same convergence rates as that in \cite{cao15}. 

For analysis of longitudinal data with partial linear models, \cite{fan2004} developed statistical estimation and inference under the corrected specified model whereas we are working with a mis-specified model. \cite{qian11} proposed a centering approach for analysis of classic longitudinal data assuming the model is correctly specified while we are dealing with omitted variable bias and model mis-specification. Moreover, the analysis of mixed synchronous and asynchronous longitudinal covariates and omitted variable analysis for longitudinal data have not been studied before.

The rest of the paper is organized as follows. In Section 2, we elaborate conditions for consistency of the na\"ive estimation of omitting the asynchronous longitudinal covariates. We then propose a partial linear model and a centering approach for consistent estimation of the regression coefficient of the synchronous longitudinal covariate and study the sampling properties of the procedure. In Section 3, we consider a \textcolor{black}{two-step} estimator of the regression coefficient of the asynchronous longitudinal covariate and derive its asymptotic properties and associated inferences. In addition, we derive methods and theory for analyzing synchronous and asynchronous longitudinal covariates simultaneously. In Section 4, we conduct Monte Carlo simulation studies to examine the finite sample performance of the proposed methods. Analysis of dataset from an ADNI study illustrates the methodology in Section 5. Concluding remarks are given in Section 6. All proofs are relegated in the Supplementary Material.

\section{Estimation and inference with omitted longitudinal covariates}
\subsection{A na\"ive approach}\label{2.1}
We first look at the case where mis-specified model is na\"ively analyzed using methods from classic longitudinal data analysis omitting the asynchronous longitudinal covariates. Assume the true model is 
\begin{equation}\label{true-model}
	Y(t)  = \alpha + X(t)^T\beta + Z(t)^T\gamma + \epsilon(t),
\end{equation}
where $Y(t)$ is the longitudinal outcome, $\alpha$ is the intercept, $X(t) \in {\mathbb R}^{p}$ is the observed longitudinal covariates measured synchronously with $Y(t),$ $Z(t) \in {\mathbb R}^{q}$ is the omitted longitudinal covariates, which may be measured asynchronously with $Y(t)$ and $X(t),$ $\beta \in {\mathbb R}^p$ and $\gamma \in {\mathbb R}^q$ are unknown parameters to be estimated, and $\epsilon(t)$ is a mean $0$ stochastic process, uncorrelated with $X(t)$ and $Z(t)$. This is a marginal model, which specifies that the conditional mean of the longitudinal response only depends on the current value of the longitudinal covariates and there is no lagged effect of the longitudinal covariates.   
In this subsection, our interest is on inference about the regression coefficient ${\beta}.$ Since $Z(t)$ is omitted, in practice, we fit the mis-specified model
\begin{equation}\label{mis-model}
	Y(t) = \alpha^\diamond + X(t)^T\beta^\diamond + \epsilon^\diamond(t),
\end{equation}
where $\alpha^\diamond $ is the intercept, $\beta^\diamond \in {\mathbb R}^p$ is the regression coefficient, and $\epsilon^\diamond(t)$ is a mean $0$ stochastic process, uncorrelated with $\alpha^\diamond$ and $X(t).$ This na\"ive practice can negatively impact estimation of ${\beta}$ in the true model (\ref{true-model}). 

Suppose we have a random sample of $n$ subjects and for the $i$th subject, there are $M_i$ longitudinal observations. Denote $Y_{ij} \in {\mathbb R}$ and $X_{ij} \in {\mathbb R}^p$ as the synchronous longitudinal response and covariates observed at times $t_{ij}, i =1, \ldots, n; j = 1, \ldots, M_i.$ 
We minimize the least square error under model (\ref{mis-model})
\begin{equation}
	\sum_{i=1}^{n}\sum_{j=1}^{M_i}(Y_{ij} - \alpha^\diamond - X_{ij}^T\beta^\diamond )^2. \nonumber 
\end{equation}
We have 
\begin{equation}\label{hatbeta}
	\hat{\beta}_n = \Big (\sum_{i=1}^{n}\sum_{j=1}^{M_i}X_{ij}X_{ij}^T \Big)^{-1} \sum_{i=1}^{n}\sum_{j=1}^{M_i}X_{ij}(Y_{ij}-\alpha^\diamond).  
\end{equation}
Taking the expectation, we have
\begin{align}\label{bias}
	E(\hat{\beta}_n) &= E\left\{\Big(\sum_{i=1}^{n}\sum_{j=1}^{M_i}X_{ij}X_{ij}^T \Big)^{-1} \sum_{i=1}^{n}\sum_{j=1}^{M_i}X_{ij}(Y_{ij}-\alpha^\diamond)\right\}\notag\\ 
	& =\beta + E\left\{\Big(\sum_{i=1}^{n}\sum_{j=1}^{M_i}X_{ij}X_{ij}^T \Big)^{-1}\sum_{i=1}^{n}\sum_{j=1}^{M_i}X_{ij}(\alpha +Z_{ij}^T\gamma + \epsilon_{ij}- \alpha^\diamond)\right\}\notag\\
	&=\beta + E\bigg(\Big(\sum_{i=1}^{n}\sum_{j=1}^{M_i}X_{ij}X_{ij}^T \Big)^{-1}\sum_{i=1}^{n}\sum_{j=1}^{M_i}X_{ij}\left[\alpha+E(Z_{ij})^T \gamma-\alpha^\diamond \right.\bigg. \notag\\
	&\bigg.\left.+\{Z_{ij}-E(Z_{ij})\}^T\gamma\right]\bigg)\notag\\
	&=\beta + E\left[\Big(\sum_{i=1}^{n}\sum_{j=1}^{M_i}X_{ij}X_{ij}^T \Big)^{-1} \sum_{i=1}^{n}\sum_{j=1}^{M_i}X_{ij}\left\{\alpha+E(Z_{ij})^T \gamma-\alpha^\diamond\right\}\right]\notag\\
	&+E \left[\Big(\sum_{i=1}^{n}\sum_{j=1}^{M_i}X_{ij}X_{ij}^T \Big)^{-1} \sum_{i=1}^{n}\sum_{j=1}^{M_i}X_{ij} \left\{Z_{ij}-E(Z_{ij})\right\}^T\gamma\right]\notag\\
	&=\beta + I + II.
\end{align} 
In (\ref{bias}), $I$ and $II$ in general do not vanish and $\hat{\beta}$ from (\ref{hatbeta}) is biased. We next show that the na\"ive approach can still work under the following conditions. 

(C1) $E\{ Z(t)\}$ is constant $\forall t.$

(C2) $X(t)$ and $Z(t)$ are uncorrelated $\forall t$. 

\begin{theorem}\label{consis}
	Under (C1) and (C2), estimation of $\beta$ in (\ref{true-model}) under the mis-specified model (\ref{mis-model}) is unbiased. 
\end{theorem}
We remark that the intercept in (\ref{true-model}) cannot be consistently estimated under the mis-specified model (\ref{mis-model}). The proof of Theorem \ref{consis} is relegated in the Supplementary Material. Furthermore, we corroborate Theorem \ref{consis} empirically by simulation studies in Section 4.

\subsection{A partial linear model approach}

For general cases, instead of working with (\ref{mis-model}), we propose to use a partial linear model
\begin{equation}\label{plm}
	Y(t) = \alpha(t) + X(t)^T \beta_p + \epsilon_p(t),
\end{equation}
where $Y(t)$ is the longitudinal response, $\alpha(t)$ is the non-parametric intercept, $X(t) \in {\mathbb R}^p$ is the longitudinal covariate, $\beta_p \in {\mathbb R}^p$ is the regression coefficient and $\epsilon_p(t)$ is a mean $0$ stochastic process, uncorrelated with $\alpha(t)$ and $X(t)$. As shown below, fitting the mis-specified model (\ref{plm}) permits unbiased estimation of $\beta$ in (\ref{true-model}) under weaker conditions than those for the na\"ive estimator in Subsection \ref{2.1}. 

We first define some notations. For the $i$th subject, we observe longitudinal response and covariates 
$\{Y_i(T_{ij}), X_i(T_{ij}) \}, j =1, \ldots, M_i,$
where $T_{ij}, j =1, \ldots, M_i,$ are the observation times for the longitudinal measurements, where $M_i$ is finite with probability $1.$ We use a counting process to represent the observation times. Specifically, 
$N_i(t) = \sum_{j=1}^{M_i}I(T_{ij} \le t)$ counts the number of longitudinal observations up to $t$ 
\citep{cao15,lin01}. We use $t_{ij}$ to denote realized value of $T_{ij}.$

For the estimation of $\alpha(t)$ in (\ref{plm}), \cite{fan2004} proposes to use local linear approximation. Specifically, for $t$ in a neighborhood of $t_0,$ by Taylor expansion, we have
\begin{equation}
	\alpha(t) \approx \alpha(t_0) + \dot{\alpha}(t_0)(t-t_0):= a_0 + a_1(t-t_0), \nonumber 
\end{equation}
where the superscript dot denotes the first order derivative. 
Let $K(\cdot)$ be a kernel function and let $h$ be a bandwidth. We aim to find 
$(\hat{a}_0, \hat{a}_1)$ minimizing
\begin{equation}\label{gee}
	\sum_{i=1}^{n}\sum_{j=1}^{M_i}K_h(t_{ij} - t_0)\{Y_i(t_{ij}) - a_0 - a_1(t_{ij} -t_0) -X_i(t_{ij})^T \beta_p \}^2,
\end{equation}
where $K_h(\cdot) = h^{-1}K(\cdot/h).$ The motivation for this is that we would like to write $\hat{a}_0$ as a function of $\beta_p.$ This is different from the GEE with working independence covariance matrix and the inverse of $K_h(t_{ij}-t_0)$ as the variance function, where interest lies in estimating $a_0, a_1$ and $\beta_p$ jointly \citep{liangzeger86}. 

From (\ref{gee}), we have 
\begin{align}\label{local-linear}
	&\hat{a}_0 =\notag\\
	& \frac{\sum_{i=1}^n \int K_h(t-t_0)\{q_2(t-t_0)-(t-t_0)q_1(t-t_0)\}\{Y_i(t)-X_i(t)^T\beta_p\}dN_i(t)}{\sum_{i=1}^n \int K_h(t-t_0)\{q_2(t-t_0)-(t-t_0)q_1(t-t_0)\} dN_i(t)},
\end{align}
where 
\begin{align*}
	q_l(t-t_0)&=\sum_{i=1}^n\int K_h(t-t_0)(t-t_0)^l dN_i(t), \quad l = 1,2.
\end{align*}

Note that $\hat{a}_0$ is linear in $Y(t) - X(t)^T \beta_p,$ so we can write the estimator of $\beta_p$ in a closed form. This is accomplished through concatenating the longitudinal measurements from the first subject to the last subject into a long vector. Specifically, 
denote $m = \sum_{i=1}^{n}M_i,$ $T^* = (T^*_1, \ldots, T^*_m):= (t_{11}, \ldots, t_{nM_n})^T .$ 
We concatenate other functions of $T^*.$ Denote $\epsilon^* :=\epsilon(T^*)= \{\epsilon(T_1^*), \ldots, \epsilon(T_m^*) \}^T,$ $\alpha^*:=\alpha(T^*)= \{\alpha(T^*_1), \ldots, \alpha(T^*_m)\}^T,$ $X^* =\{X_1^*, \ldots, X^*_m \} := \{X_1(t_{11}),  \ldots, \newline X_n(t_{nM_n}) \}^T$ and $Y^*= \{Y^*_1, \ldots, Y^*_m \}^T := \{ Y_1(t_{11}), \ldots, Y_n(t_{nM_n})\}^T.$ Then $\hat{\alpha}^* = S(Y^* - X^* \beta_p),$ where $Y^* - X^* \beta_p$ is an $m \times 1$ column vector, $S$ is an $m \times m$ symmetric matrix with the $i$th row and $j$th column entry $s_{ij} = w_{ij}(\sum_{j=1}^{m}w_{ij})^{-1},$ where $w_{ij} = K_h(T^*_i - T^*_j)\{q_{i,2} - (T^*_i - T^*_j)q_{i,1}\},$ where $q_{i,l} = \sum_{j=1}^{m}K_h(T^*_i - T^*_j)(T^*_i - T^*_j)^l, l =1, 2.$ As a result, $\hat{\alpha}^*$ is an $m \times 1$ column vector. Marginally, substituting $\hat{\alpha}^*$ into 
$$
Y^* = \alpha^* + X^* \beta_p+ \epsilon^*,
$$
we obtain 
$$
(I-S)Y^* = (I-S)X^*\beta_p + \epsilon^*,
$$
where $I$ is the identity matrix of dimension $m.$ By minimizing the squared error, we have
\begin{equation}
	\hat{\beta}_p = \{X^{*T}(I-S)^T(I-S)X^{*}\}^{-1}X^{*T}(I-S)^T(I-S)Y^*. \nonumber
\end{equation}
With $\hat{\beta}_p$ estimated, we obtain residual vector $\hat{\epsilon}_i$ for the $i$th subject from fitting (\ref{plm}), $i = 1, \ldots, n.$ Define $\hat{C}=\mbox{diag}\{ \hat{\epsilon}_1\hat{\epsilon}_1^T, \ldots, \hat{\epsilon}_n \hat{\epsilon}_n^T\},$ and $\hat{V} =X^{*T}(I-S)^T\hat{C}(I-S)X^*.$ The variance of $\hat{\beta}_p$ can be estimated by $\widehat{\mbox{Var}}(\hat{\beta}_p) = D^{-1}\hat{V}D^{-1},$ where $D = X^{*T}(I-S)^T(I-S)X^*.$ 
We next show theoretical properties of $\hat{\beta}_p.$ In \cite{fan2004}, it has been shown that $\hat{\beta}_p$ is a consistent estimator of $\beta_p$ in (\ref{plm}). 
We shall show that $\hat{\beta}_p$ is consistent for $\beta$ in (\ref{true-model}) and derive its asymptotic distribution. We need the following assumptions.

(C3) $N_i(t)$ is independent of $\{X_i(t), Z_i(t)\}$ for each $t$ and $E\{ dN_i(t)\} = \lambda(t)dt, i=1, \ldots, n.$ Moreover, $\epsilon(t)$ is a mean $0$ process, uncorrelated  with $X(t)$ and $Z(t).$

(C4) $E\int  \{\tilde{X}(t) \tilde{X}(t)^T \}\lambda(t)dt$ is positive definite, where $\tilde X(t)=X(t)-E\{X(t)\}.$ 

(C5) $K(\cdot)$ is a symmetric density function satisfying $\int z K(z)dz = 0, \int z^2 \newline K(z)dz <\infty$ and $\int K(z)^2dz < \infty.$

(C6) $h\rightarrow 0$ and $nh\rightarrow \infty.$

Condition (C3) requires that the observation process and error process are independent of the longitudinal covariate processes. Condition (C4) ensures identifiability. 
Conditions (C5) and (C6) specify valid kernels and  bandwidths. 

The following theorem, which is established in the Supplementary Material, states the asymptotic properties of $\hat{\beta}_p.$
\begin{theorem}\label{thm1}
	Under (C2)-(C6), the asymptotic distribution of $\hat{\beta}_p$ satisifies
	\begin{equation}\label{asym-plm}
		\sqrt{n}(\hat\beta_p-\beta_0) \overset{d}{\rightarrow} N(0,{ A_{\textcolor{black}{\beta_0}}}^{-1}{\Sigma_{\textcolor{black}{\beta_0}}} {A_{\textcolor{black}{\beta_0}}}^{-1}),
	\end{equation} 
	where $\beta_0$ is the true value of $\beta$ in (\ref{true-model}),
	\begin{eqnarray*}
		{A_{\textcolor{black}{\beta_0}}}&=&E\int \{\tilde X(t)\tilde X(t)^T\}\lambda(t)dt,\quad \mbox{and}\\
		{\Sigma_{\textcolor{black}{\beta_0}}}&=&E \Big\{\int \tilde{X}(t) \{\tilde{Z}(t)^T\gamma_0 + \epsilon(t) \}dN(t)\Big\}^{\otimes2}.
	\end{eqnarray*}
\end{theorem}
This result is different from that in \cite{fan2004} as it is derived under the working model (\ref{plm}) to make inference of parameters in the true model (\ref{true-model}). 

We further make connections between the asymptotic variance of $\hat\beta_p$, $\Sigma = {A_{\textcolor{black}{\beta_0}}}^{-1}{\Sigma_{\textcolor{red}{\beta_0}}}{A_{\textcolor{red}{\beta_0}}}^{-1}$, and its estimator, $D^{-1}\hat VD^{-1}$. Specifically, $(I-S)X^*$ can be viewed as a realization of $\tilde X$, and its inner product $D=X^{*T}(I-S)^T(I-S)X^*$ reflects ${A_{\textcolor{red}{\beta_0}}}$, the integral with respect to the intensity function of the counting process $\lambda(t)dt$. Additionally $\hat{\epsilon}_i$ of $\hat C$ in $\hat V$ corresponds to $\tilde Z_i(t)^T\gamma_0+\epsilon_i(t)$ in $\Sigma_\beta$ as both are based on the residuals with $X_i(t)$ as the covariate, $i=1, \ldots, n.$

\subsection{A centering approach}

The partial linear model approach uses a non-parametric function to model the omitted longitudinal covariate $Z(t)$ and its effect. A different idea is to eliminate $Z(t)$ from the model by centering $X(t)$ and $Y(t).$ Specifically, we first take the unconditional expectation of (\ref{true-model}):
\begin{equation}
	E\{ Y(t)\} = \alpha + E\{X(t) \}^T\beta + E\{Z(t)\}^T \gamma, \nonumber
\end{equation}
and subtract it from (\ref{true-model}):
\begin{equation}\label{transform-model}
	E\{Y(t) \mid X(t), Z(t)\} - E \{Y(t) \} = \Big[ X(t) -E\{X(t) \}\Big]^T \beta + \Big[Z(t) -E\{Z(t) \} \Big]^T\gamma.
\end{equation}
By (C2), taking expectation conditional on $X(t)$ only, (\ref{transform-model}) becomes
\begin{equation}\label{center-model}
	E\{\tilde{Y}(t) \mid \tilde{X}(t)\} = \tilde{X}(t)^T\beta, 
\end{equation}
where $\tilde{Y}(t) = Y(t) - E\{ Y(t)\}.$ The estimate of $\beta$ can be obtained through the usual linear model analysis. Our motivation and set up are very different from \cite{qian11}, where a similar centering approach is developed for analysis of classic longitudinal data. 

Note that in (\ref{center-model}), the unknown mean processes $E\{Y(t)\}$ and $E\{ X(t)\}$ need to be estimated. 
This can be achieved through Nadaraya-Watson estimator \citep{nadaraya64,watson64}. Denote $m_Y(t) = E\{ Y(t)\}$ and $m_X(t)=E\{X(t)\}.$ We have
\begin{align*}
	\hat m_Y(t_0)&=\frac{\sum_{i=1}^n\int K_h(t-t_0)Y_i(t)d N_i(t)}{\sum_{i=1}^n\int K_h(t-t_0)dN_i(t)}  \quad \mbox{and} \\
	\hat m_X(t_0)&=\frac{\sum_{i=1}^n\int K_h(t-t_0)X_i(t)d N_i(t)}{\sum_{i=1}^n\int K_h(t-t_0)d N_i(t)}, 
\end{align*}
where $K_h(\cdot) = h^{-1}K(\cdot/h),$ $K(\cdot)$ is a kernel function and $h$ is the bandwidth. Let $\hat{Y}_i(t) = Y_i(t) - \hat{m}_Y(t)$ and $\hat{X}_i(t) = X_i(t) - \hat{m}_X(t).$ The estimating equation for $\beta$ is
\begin{equation}\label{ee}
	U(\beta) = n^{-1}\sum_{i=1}^{n}\int \hat{X}_i(t)\{ \hat{Y}_i(t) - \hat{X}_i(t)^T\beta\}dN_i(t).
\end{equation}
Solving (\ref{ee}), we obtain 
\begin{equation}
	\hat{\beta}_c = \Big \{\sum_{i=1}^{n} \int \hat{X}_i(t)\hat{X}_i(t)^TdN_i(t) \Big\}^{-1} \sum_{i=1}^{n}\int \hat{X}_i(t)\hat{Y}_i(t)dN_i(t). \nonumber 
\end{equation}
We shall show that $\hat{\beta}_c$ is a consistent estimator of the true $\beta$ in (\ref{true-model}) and establish its limiting distribution. 
Note that
$$
\hat{\beta}_c - \beta_0 = \Big \{\sum_{i=1}^{n} \int \hat{X}_i(t)\hat{X}_i(t)^TdN_i(t) \Big\}^{-1} \sum_{i=1}^{n}\int \hat{X}_i(t)\{ \hat{Y}_i(t)-\hat{X}_i^T\beta_0 \}dN_i(t).
$$
Therefore, the variance of $\hat{\beta}_c$ can be estimated with the sandwich formula:
\begin{eqnarray*}
	\widehat{\mbox{Var}}(\hat{\beta}_c) &=& \{\sum_{i=1}^{n}\int\hat{X}_i(t)\hat{X}_i(t)^TdN_i(t) \}^{-1} \\
	&&\sum_{i=1}^{n}\Big[ \int \hat{X}_i(t)\{\hat{Y}_i(t)-\hat{X}_i(t)^T\hat{\beta}_c\} dN_i(t)\Big]^{\otimes 2}\\
	&&\{\sum_{i=1}^{n}\int\hat{X}_i(t)\hat{X}_i(t)^TdN_i(t) \}^{-1}.
\end{eqnarray*}

We need an additional smoothness assumption specified in (C7) below.

(C7) $E\{X(t) \}$ and $E\{Y(t) \}$ are continuous functions for any $t.$ 

The following theorem, which is established in the Supplementary Material, states the asymptotic properties of $\hat{\beta}_c.$
\begin{theorem} \label{thm2}
	Under (C2)-(C7), the asymptotic distribution of $\hat{\beta}_c$ satisfies
	\begin{equation}
		\sqrt{n}(\hat{\beta}_c - \beta_0) \overset{d}{\rightarrow} N(0, {A_{\textcolor{black}{\beta_0}}}^{-1}{\Sigma_{\textcolor{black}{\beta_0}}} {A_{\textcolor{black}{\beta_0}}}^{-1}),
	\end{equation}
	where ${A_{\textcolor{black}{\beta_0}}}$ and ${\Sigma_{\textcolor{black}{\beta_0}}}$ are the same as those in Theorem \ref{thm1}.
\end{theorem}

We note that $\hat{\beta}_p$ and $\hat{\beta}_c$ are asymptotically unbiased, obtain parametric root $n$ rate of convergence and have the same limiting variance. This suggests that the newly proposed two estimators should perform similarly in practice. Simulation studies reported in Section 4 further substantiate these theoretical findings. It is counter-intuitive that we get an efficient estimation of $\beta_0$ with less information, as information contained in $Z(t)$ are not used. This is due to the key assumption (C2). 
When (C2) is violated, we conjecture similar strategies of FWL theorem would work, with slower rate of convergence due to the asynchronous nature of $Z(t).$ 

Both partial linear model approach and centering approach require a bandwidth for smoothing. We experimented various values of bandwidths in the allowable range in the simulation studies and the results are fairly robust to the choice of bandwidth. In practice, one may use cross validation to select bandwidth. In terms of computation, centering is faster. The trade off is that partial linear model approach allows us to get a sense of the omitted $Z(t)$ through the estimated non-parametric intercept term. 
\section{Estimation and inference of asynchronous longitudinal covariates} 
\subsection{A two-step method}
In this section, we consider the case that longitudinal covariates $X(t)$ and $Z(t)$ are asynchronous and longitudinal response $Y(t)$ is observed in alignment with $X(t).$ There is no existing literature to deal with such mixed synchronous and asynchronous longitudinal covariates. To be specific, suppose we have a random sample of $n$ subjects. For subject $i=1, \ldots, n$,
$N_i(t,s) = \sum_{j=1}^{M_i}\sum_{k=1}^{L_i}I(t_{ij} \le t, s_{ik} \le s)$
counts the number of observation times up to $t$ on $X(\cdot)$ and $Y(\cdot)$ and up to $s$ on $Z(\cdot),$ where $t_{ij}, j =1, \ldots, M_i$ are the observation times of $X(\cdot)$ and $Y(\cdot)$ and $s_{ik}, k =1, \ldots, L_i$ are the observation times of $Z(\cdot).$ Denote $E\{d N_i(t,s)\} = \eta(t,s)dtds, i =1, \ldots, n.$
We propose a two-step approach to estimate $\beta$ and $\gamma$ in (\ref{true-model}). 

Step 1: Regress longitudinal response $Y(t)$ on synchronous longitudinal covariate $X(t)$ to get $\hat{\beta}$ and the residuals. 

Step 2. Regress residuals from Step 1 on asynchronous longitudinal covariate $Z(t)$ to estimate $\hat{\gamma}$. 

In Step 1, either partial linear model approach or centering approach can be used as they have the same asymptotic distribution. Once $\hat{\beta}$ is obtained, we compute the residual $\hat{\omega}_i(T_{ij}) = Y_i(T_{ij}) - X_i(T_{ij})^T\hat{\beta}.$ 
In Step 2, to estimate $\gamma,$
we propose the following estimating equation \citep{cao15}
\begin{equation}\label{ee-asyn}
	U^f (\gamma) = n^{-1}\sum_{i=1}^{n}\iint K_h(t-s)Z_i(s)\{Y_i(t) - Z_i(s)^T\gamma - X_i(t)^T\hat{\beta} \}dN_i(t,s),
\end{equation}
where $K_h(t) = K(t/h)/h, K(t)$ is a symmetric kernel function, usually taken to be the Epanechnikov kernel $K(t) = 0.75(1-t^2)_{+}$ and $h$ is the bandwidth.
Solving $U^f(\gamma) = 0,$ we obtain
\begin{align*}
	\hat\gamma&=\left\{\sum_{i=1}^n\iint K_h(t-s) Z_i(s) Z_i(s)^{T}d N_i(t,s)\right\}^{-1}\\
	&\times \sum_{i=1}^n\iint K_h(t-s) Z_i(s)\{ Y_i(t)- X_i(t)^T\hat\beta \}dN_i(t,s).
\end{align*}

The idea of the two-step approach is intuitive. Note (\ref{true-model}) can be written as $Y(t) - X(t)^T\beta = Z(t)^T\gamma + \epsilon(t),$ where we abuse notation by absorbing intercept $\alpha$ into $\beta$ and letting the first entry of vector $X_{ij}$ to be 1. Once we get $\hat{\beta},$ the estimation of $\gamma$ can proceed as an asynchronous regression problem with residual $Y(t) - X(t)^T\hat{\beta}$ as the new response. 

We next present asymptotic properties of $\hat{\gamma}.$ Denote $\sigma^2(t) = \mbox{Var}\{ \epsilon(t)\}$ and let $\gamma_0$ be the true regression coefficient. We need the following conditions.

(C8) $\eta(t,s)$ is twice continuously differentiable for $(t,s) \in [0,1]^{\otimes 2}.$ Moreover, For $t_1 \ne t_2,s_1\ne s_2$, $P \{dN(t_1,s_1) = 1 \mid N(t_2,s_2) - N(t_2-,s_2-) =1\} = f(t_1, t_2,s_1,s_2)dt_1d s_1$, where $f(t_1,t_2,s_1,s_2)$ is continuous for $t_1 \ne t_2,s_1\ne s_2$, and $f\{t_1\pm, t_2\pm,s_1\pm,s_2\pm\}$ exists.

(C9) $E\{Z(t)Z(s)^T \}$ is twice continuously differentiable for $(t,s) \in [0,1].$ In addition, \newline
$\int E \{Z(s) Z(s)^T \} \eta(s,s)ds$ is positive definite and
$$
\|\int E\{Z(s) Z(s)^T \}\eta(s,s)\sigma^2(s)ds\|_{\infty} < \infty,
$$
where for a square matrix $A,$  $\|A\|_{\infty} = \max_{1 \le i \le n}\sum_{j=1}^n |a_{ij}|.$

(C10) $nh\rightarrow \infty$ and $nh^5 \rightarrow 0.$ 

The following theorem states the asymptotic properties of $\hat{\gamma}.$
\begin{theorem} \label{thm3}
	Under (C3)-(C5),(C7)-(C10), the asymptotic distribution of $\hat{\gamma}$ satisfies 
	\begin{equation}
		\sqrt{nh}(\hat{\gamma} - \gamma_0) \overset{d}{\rightarrow}N(0, A_{\textcolor{black}{\gamma_0}}^{-1}\Sigma_{\textcolor{black}{\gamma_0}} A_{\textcolor{black}{\gamma_0}}^{-1}),
	\end{equation}
	where
	\begin{align*}
		A_{\textcolor{black}{\gamma_0}}&=\int E\left\{ Z(s)Z(s)^T\right\}\eta(s,s)ds\quad\text{and}\\
		\Sigma_{\textcolor{black}{\gamma_0}}&= \int K(z)^2dz\int E\left\{Z(s)Z(s)^T\right\} \eta(s,s)\sigma^2(s) ds.
	\end{align*}
\end{theorem}
The asymptotic distribution of $\hat{\gamma}$ is the same as that in \cite{cao15} with identity link function. As $\hat{\gamma}$ has $\sqrt{nh}$ rate of convergence, slower than the $\sqrt{n}$ rate of convergence of $\hat{\beta},$ plugging in $\hat{\beta}_p$ or $\hat{\beta}_c$ does not affect the limiting distribution of $\hat{\gamma},$ which corroborates the validity of the proposed two-step method. That is, estimation of $\gamma$ is as efficient as if $\beta$ were known {\it a priori}. The variance of $\hat{\gamma}$ can be estimated from the sandwich formula
\begin{align*}
	\widehat{\mbox{Var}}(\hat{\gamma})&=\left\{\sum_{i=1}^n\iint K_h(t-s) Z_i(s) Z_i(s)^{T}d N_i(t,s)\right\}^{-1}\\
	& \sum_{i=1}^n\left [\iint K_h(t-s) Z_i(s)\{ Y_i(t)-Z_i(s)^T \hat{\gamma} -X_i(t)^T\hat\beta \} dN_i(t,s) \right]^{\otimes 2}\\
	&\left\{\sum_{i=1}^n\iint K_h(t-s) Z_i(s) Z_i(s)^{T}d N_i(t,s)\right\}^{-1}.
\end{align*}

\subsection{Simultaneous estimation of synchronous and asynchronous longitudinal covariates}
For the mixed synchronous and asynchronous longitudinal covariates, a natural idea is to use estimating equations to get estimators of $\beta$ and $\gamma$ simultaneously, similar to \cite{cao15}. 
Denote $W(t,s) = \{X(t)^T, Z(s)^T \}^T.$ We use the estimating equation
\begin{align}
	&U_w(\beta, \gamma)\notag \\
	&=n^{-1}\sum_{i=1}^{n}\iint K_h(t-s)W_i(t,s)\{Y_i(t) - Z_i(s)^T\gamma - X_i(t)^T{\beta} \}dN_i(t,s)  \label{ee2-1}
\end{align}
to get $\hat{\beta}_w$ and $\hat{\gamma}_w,$ respectively. 
As shown in Theorem \ref{Ch4: cor cao}, the resulting estimators are both asymptotically unbiased.
However, rate of convergence of the obtained estimator of $\beta$ is slower than the parametric $\sqrt{n}$ rate, as obtained through the partial linear model or centering approach. The rationale is that unnecessary smoothing on $X(t)$ makes the corresponding estimator less efficient. This is further demonstrated in the simulation studies. We need the following assumption. 

(C11) $\E\{W(t,s)W(t,s)^T\} \in \mathbb{R}^{(p+q)\times(p+q)}$ is twice continuously differentiable for $(t,s) \in [0,1]$ with $W(t,s)=\{X(t)^T,Z(s)^T\}^T.$ $\int \E \{W(t,t)W(t,t)^T\} \newline \eta(t,t)\df t$ is positive definite and
$$
\|\int\E\{W(t,t)W(t,t)^T\}\sigma^2(t)\eta(t,t)\df t \|_{\infty}< \infty \quad \forall t ,
$$
where for a square matrix $A,$ $\|A\|_\infty = \max_{1 \le i \le n}\sum_{j=1}^n |a_{ij}|.$
\begin{theorem}\label{Ch4: cor cao}
	Under conditions (C8), (C10) and (C11), let $\hat{\theta}_w = (\hat{\beta}_w^T, \hat{\gamma}_w^T)^T$ and $\theta_0=(\beta_0^T, \gamma_0^T)^T,$ the asymptotic distributions of $\hat\theta_w$ satisfies
	\begin{align*}
		\sqrt{nh}(\hat\theta_w-\theta_0 )&\stackrel{d}{\to}N(0,A_{\textcolor{black}{\theta_0}}^{-1}\Sigma_{\textcolor{black}{\theta_0}} A_{\textcolor{black}{\theta_0}}^{-1}),
	\end{align*}
	where \begin{align*}
		A_{\textcolor{black}{\theta_0}}&=\int\E\{W(t,t)W(t,t)^T\}\eta(t,t)\df t \quad\text{and}\\
		\Sigma_{\textcolor{black}{\theta_0}}&=\int K(z)^2\df z\int\E\{W(t,t)W(t,t)^T \}\sigma^2(t)\eta(t,t)\df t.
	\end{align*}
\end{theorem}

The asymptotic covariance matrix can be estimated by the sandwich formula
\begin{eqnarray*}
	&&\widehat{\mbox{Var}}(\hat{\theta}_w) \\
	&= & \left \{\sum_{i=1}^n \iint K_h(t-s)W_i(t,s)W_i(t,s) ^T dN_i(t,s) \right \}^{-1}\\
	&&\sum_{i=1}^{n}\left [\iint K_h(t-s)W_i(t,s)\{Y_i(t) - Z_i(s)^T\hat{\gamma} - X_i(t)^T\hat{\beta} \}dN_i(t,s) \right]^{\otimes 2}\\
	&&\left \{\sum_{i=1}^n \iint K_h(t-s)W_i(t,s)W_i(t,s) ^T dN_i(t,s) \right \}^{-1}.
\end{eqnarray*}

It is worth noting that Theorem~\ref{Ch4: cor cao} requires weaker condition on the relationship between $X(t)$ and $Z(t)$ than the two-step approach. For the latter to work, we need the key assumption that $X(t)$ and $Z(t)$ are uncorrelated as specified in (C2). This reflects the trade off between robustness and efficiency.   

\subsection{Bandwidth selection}\label{Ch3: subsec bdwd}
Our approach in the estimation of $\gamma$ depends on the bandwidth selection. In synchronous longitudinal data, cross-validation is usually used to select the optimal bandwidth by minimizing the squared prediction error. However, in asynchronous longitudinal data, since observations are mismatched, prediction errors are not well defined. \cite{cao15} proposes to minimize the mean squared error by calculating bias and variance separately. First, based on the asymptotic result, bias is of the same order as the bandwidth square. One can regress the squared bandwidth with the estimated regression coefficient to obtain the slope estimate. The bias is approximated by the multiplication of the slope estimate and the squared bandwidth. Second, the data are split into two halves and coefficient estimates are obtained for each half. The squared difference of the two coefficient estimates divided by $4$ is used to approximate the variance. The mean squared error is squared bias plus variance and the optimal bandwidth is chosen to be the one that minimizes the mean squared error. This method is somewhat {\it ad hoc} since only two folds of data are used and squared difference is coarse and may be imprecise estimate of the variance. 

We propose a new kernel smoothed cross-validation to select the optimal bandwidth within certain range. First, we split the data into several folds, and estimate regression coefficient without one fold. In computing the prediction error, we use kernel smoothing to deal with the mismatched response and covariate. Specifically, suppose $\hat{\beta}^{(-k)}$ and $\hat{\gamma}^{(-k)}$ are estimates without the $k$th fold. The squared prediction error for the $k$th fold is computed as 
$$\frac{\sum_{i=1}^{n^{(k)}}\iint K_h(t-s)\{Y_i(t)-X_i(t)^T\hat\beta^{(-k)}-Z_i(s)^T\hat\gamma^{(-k)}\}^2d N_i(t,s)}{\sum_{i=1}^{n^{(k)}}\iint K_h(t-s)dN_i(t,s)},$$
where $n^{(k)}$ is the number of subjects in the $k$th fold. We take average of the squared prediction errors over all folds, and select the bandwidth with the smallest average squared prediction error. A typical functional relationship between the average squared prediction error and bandwidth is depicted in Figure \ref{Ch3: fig mse}.
\begin{figure}[!ht]
	\includegraphics[width = 4.5in]{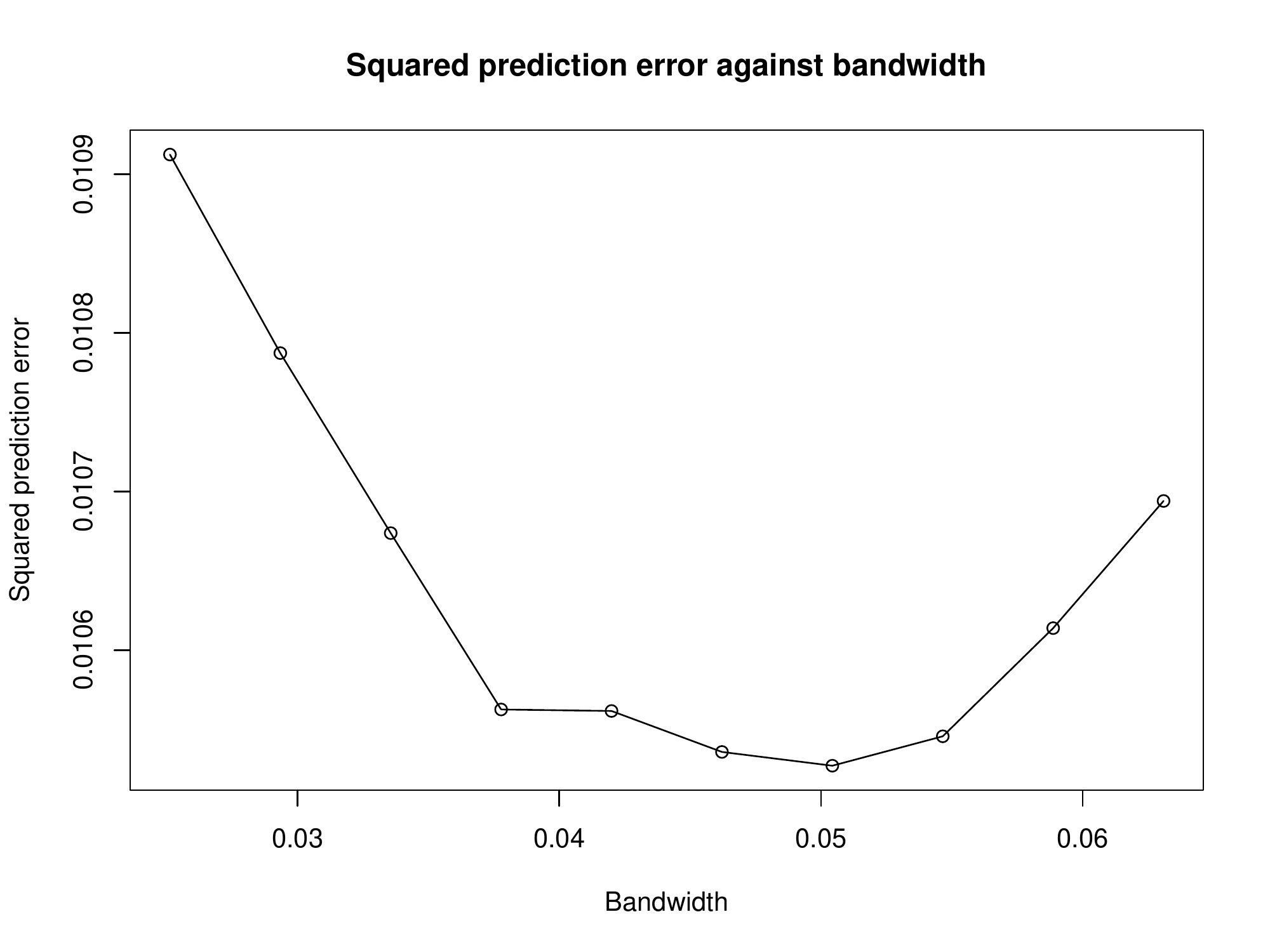}
	\centering \caption{Typical squared prediction error against bandwidth for $n=100, E\{Z(t)\}=2\sin(2\pi t)$, with bandwidth ranging from $n^{-0.8}$ to $n^{-0.6}$. }\label{Ch3: fig mse}
\end{figure}

\section{Numerical studies}
In this section, we investigate the finite sample performance of the proposed estimators through Monte Carlo simulations.

\subsection{Omitted longitudinal covariate}\label{Ch3: subsec numerical omit} 
We first examine the performance of $\hat{\beta}_p$ and $\hat{\beta}_c$ along with the na\"ive estimator $\hat{\beta}_n$ when some important covariates are omitted. The model we use to generate data is
$$
Y(t) = \alpha + X(t)^T\beta + Z(t)^T\gamma + \epsilon(t),
$$
where $\alpha = 1, \beta = 2$ and $\gamma =-1.$ We generate $1,000$ datasets, each consisting of $n=100, 400,$ or $900$ subjects. Bandwidth is fixed at $n^{-0.6};$ other bandwidths yield similar results that are relegated in the Supplementary Material. The number of observations for each subject is $\mbox{Poisson} (5) + 1$ and the observation times are generated from the uniform distribution ${\cal U}(0,1).$ The covariate process $X(t)$ and $Z(t)$ are both Gaussian, with $E\{X(t) \} = \sqrt{t}, \mbox{Cov}\{X(t), X(s) \} = \mbox{Cov}\{Z(t), Z(s) \} = e^{-|t-s|}$ and $E\{Z(t) \} = 0.5 + t, 0.5 + t^2, 2\mbox{sin}(2\pi t)$ or $E\{ Z(t)\} =2.$ The error process $\epsilon(t)$ is Gaussian with $E\{ \epsilon(t)\} =0$ and $\mbox{Cov}\{\epsilon(t), \epsilon(s) \} = 2^{-|t-s|}.$ $X(t), Z(t)$ and $\epsilon(t)$ are independently generated. The results are summarized in Table \ref{Ch3: table beta}. 
\begin{table}[!ht]
	\caption{1000 simulation results for inference of $\beta$ with $h=n^{-0.6}$}\label{Ch3: table beta}
	\begin{center}\small
		\scalebox{0.75}{\begin{tabular}{lrrrrrrrrrrrr}
				\hline
				& \multicolumn{4}{c}{Na\"ive} & \multicolumn{4}{c}{PLM} & \multicolumn{4}{c}{Centering}\\
				&	Bias	&	SD	&	SE	&	CP	&	Bias	&	SD	&	SE	&	CP	&	Bias	&	SD	&	SE	&	CP\\
				\hline
				\multicolumn{3}{l}{ Independent covariates} \\
				\multicolumn{3}{l}{ $E\{Z(t)\}=2$} \\
				$n=100$ & $0.002$ & $0.117$ & $0.113$ & $94$ & $0.003$ & $0.122$ & $0.116$ & $93$ & $0.003$ & $0.123$ & $0.117$ & $93$ \\ 
				$n=400$ & $0.001$ & $0.060$ & $0.058$ & $94$ & $0.002$ & $0.063$ & $0.060$ & $94$ & $0.002$ & $0.063$ & $0.060$ & $94$ \\ 
				$n=900$ & $0.0004$ & $0.039$ & $0.039$ & $94$ & $0.001$ & $0.041$ & $0.040$ & $94$ & $0.001$ & $0.041$ & $0.040$ & $94$ \\ 
				\multicolumn{3}{l}{ $E\{Z(t)\}=0.5+t$} \\
				$n=100$ & $-0.062$ & $0.122$ & $0.115$ & $89$ & $0.001$ & $0.125$ & $0.117$ & $92$ & $0.001$ & $0.125$ & $0.118$ & $92$ \\ 
				$n=400$ & $-0.066$ & $0.057$ & $0.058$ & $79$ & $-0.003$ & $0.059$ & $0.060$ & $95$ & $-0.003$ & $0.059$ & $0.060$ & $95$ \\ 
				$n=900$ & $-0.064$ & $0.038$ & $0.039$ & $62$ & $-0.001$ & $0.040$ & $0.040$ & $95$ & $-0.001$ & $0.040$ & $0.040$ & $95$ \\ 
				\multicolumn{3}{l}{ $E\{Z(t)\}=0.5+t^2$} \\
				$n=100$ & $-0.058$ & $0.116$ & $0.114$ & $92$ & $0.003$ & $0.120$ & $0.116$ & $94$ & $0.003$ & $0.120$ & $0.116$ & $94$ \\ 
				$n=400$ & $-0.062$ & $0.057$ & $0.058$ & $81$ & $-0.002$ & $0.059$ & $0.060$ & $96$ & $-0.002$ & $0.059$ & $0.060$ & $96$ \\ 
				$n=900$ & $-0.058$ & $0.039$ & $0.039$ & $68$ & $0.002$ & $0.040$ & $0.040$ & $95$ & $0.002$ & $0.040$ & $0.040$ & $95$ \\ 
				\multicolumn{3}{l}{ $E\{Z(t)\}=2\mbox{sin}(2\pi t)$} \\
				$n=100$ & $0.228$ & $0.140$ & $0.133$ & $57$ & $-0.005$ & $0.126$ & $0.117$ & $92$ & $-0.006$ & $0.126$ & $0.117$ & $92$ \\ 
				$n=400$ & $0.229$ & $0.067$ & $0.068$ & $7$ & $0.0003$ & $0.061$ & $0.060$ & $95$ & $0.0003$ & $0.061$ & $0.060$ & $95$ \\ 
				$n=900$ & $0.229$ & $0.045$ & $0.045$ & $0$ & $-0.0004$ & $0.040$ & $0.040$ & $96$ & $-0.0004$ & $0.040$ & $0.040$ & $96$ \\ 
				\hline
				\multicolumn{3}{l}{ Uncorrelated covariates} \\
				\multicolumn{3}{l}{ $E\{Z(t)\}=2$} \\
				$n=100$ & $-0.002$ & $0.207$ & $0.189$ & $91$ & $-0.003$ & $0.219$ & $0.196$ & $90$ & $-0.003$ & $0.219$ & $0.196$ & $90$ \\ 
				$n=400$ & $0.003$ & $0.106$ & $0.103$ & $94$ & $0.002$ & $0.112$ & $0.107$ & $93$ & $0.002$ & $0.112$ & $0.107$ & $93$ \\ 
				$n=900$ & $0.002$ & $0.073$ & $0.070$ & $94$ & $0.002$ & $0.076$ & $0.073$ & $93$ & $0.002$ & $0.076$ & $0.073$ & $93$ \\ 
				\multicolumn{3}{l}{ $E\{Z(t)\}=0.5+t$} \\
				$n=100$ & $-0.053$ & $0.207$ & $0.188$ & $91$ & $0.016$ & $0.220$ & $0.193$ & $90$ & $0.016$ & $0.220$ & $0.193$ & $90$ \\ 
				$n=400$ & $-0.067$ & $0.106$ & $0.104$ & $89$ & $-0.002$ & $0.111$ & $0.108$ & $94$ & $-0.002$ & $0.111$ & $0.108$ & $94$ \\ 
				$n=900$ & $-0.059$ & $0.073$ & $0.070$ & $85$ & $0.005$ & $0.076$ & $0.073$ & $94$ & $0.005$ & $0.076$ & $0.073$ & $94$ \\ 
				\multicolumn{3}{l}{ $E\{Z(t)\}=0.5+t^2$} \\
				$n=100$ & $-0.060$ & $0.217$ & $0.190$ & $90$ & $0.003$ & $0.227$ & $0.196$ & $90$ & $0.003$ & $0.227$ & $0.196$ & $90$ \\ 
				$n=400$ & $-0.064$ & $0.107$ & $0.103$ & $89$ & $-0.003$ & $0.113$ & $0.107$ & $93$ & $-0.003$ & $0.113$ & $0.107$ & $94$ \\ 
				$n=900$ & $-0.062$ & $0.071$ & $0.071$ & $86$ & $-0.002$ & $0.074$ & $0.074$ & $94$ & $-0.002$ & $0.074$ & $0.074$ & $94$ \\ 
				\multicolumn{3}{l}{ $E\{Z(t)\}=2\mbox{sin}(2\pi t)$} \\
				$100$ & $0.237$ & $0.226$ & $0.204$ & $74$ & $-0.002$ & $0.225$ & $0.194$ & $90$ & $-0.002$ & $0.225$ & $0.195$ & $90$ \\ 
				$400$ & $0.233$ & $0.118$ & $0.111$ & $46$ & $0.001$ & $0.115$ & $0.108$ & $93$ & $0.001$ & $0.115$ & $0.108$ & $93$ \\ 
				$900$ & $0.231$ & $0.079$ & $0.076$ & $17$ & $0.002$ & $0.076$ & $0.074$ & $94$ & $0.002$ & $0.076$ & $0.074$ & $94$ \\ 
				\hline
		\end{tabular}}
	\end{center}
	\footnotesize{Note: ``Bias" is the empirical bias, ``SD'' is the sample standard deviation, ``SE'' is the average of the standard error estimates, ``CP''$/100$ represents the coverage probability of the $95\%$ confidence interval of $\hat{\beta}$.}
\end{table}

We observe that when $E\{Z(t) \}=2,$ the na\"ive estimator $\hat{\beta}_n$ performs reasonably well. For all other time-varying $E\{Z(t)\},$ there is evidence of substantial bias and poor coverage probabilities for the true $\beta.$ The partial linear model (PLM) based estimator $\hat{\beta}_p$ and centering (Centering) approach based estimator $\hat{\beta}_c$ have almost identical performance with $\sqrt{n}$ rate of convergence. As sample size increases the empirical and model based standard errors tend to agree and the coverage is close to the nominal $95\%$ level. The performance improves with larger sample sizes. 

In order to obtain valid inference based on $\hat{\beta}_p$ and $\hat{\beta}_c,$ $X(t)$ and $Z(t)$ do not have to be independent. Being uncorrelated suffices. The simulation setup is identical to the previous scenario except that \textcolor{black}{the covariate processes }
and the error process $\epsilon(t)$ are generated in the following manner:
\begin{enumerate}
	\item Generate a Gaussian process $\nu(t)$ with $\mbox{Cov}\{\nu(t),\nu(s)\}=e^{-|t-s|}$;
	\item \textcolor{black}{$Z^{\prime}(t)= Z(t) +\nu(t)$} , where $E\{Z(t)\}=0.5+t,0.5+ t^2,0.5+\sqrt{t}$ or $E\{Z(t)\}=2$;
	\item \textcolor{black}{$X^{\prime}(t)= X(t) +\omega\nu(t)$, } 
	where $\omega\sim N(0,1)$;
	\item $\epsilon(t)=\tau\nu(t)$, where $\tau\sim N(0,1)$.
\end{enumerate}
We observe that the na\"ive method $\hat{\beta}_n$ is biased except when $E\{ Z(t)\} = 2.$ The proposed two estimators $\hat{\beta}_p$ and $\hat{\beta}_c$ remain unbiased.

\subsection{Asynchronous longitudinal covariate}
We next study coefficient estimation when $Z(t)$ is mismatched with $X(t)$ and $Y(t).$ The simulation setup is the same as that in Section \ref{Ch3: subsec numerical omit}  with independent $X(t), Z(t)$ and $\epsilon(t).$ As estimators based on partial linear model approach and centering approach have same asymptotic distribution, we illustrate the two-step method with the centering approach, denoted as Centering + KS in Table \ref{Ch3: table all}.
\begin{table}[!ht] 
	\caption{1000 simulation results for $\beta,\gamma$ and $\alpha$}\label{Ch3: table all} \footnotesize
	\setlength{\tabcolsep}{3pt}
	\begin{center}
		\scalebox{0.75}{\begin{tabular}{llrrrrrrrrrrrrrrrr}
				\hline
				& & \multicolumn{4}{c}{LVCF} & \multicolumn{4}{c}{Centering+LVCF} &
				\multicolumn{4}{c}{Centering+KS} & \multicolumn{4}{c}{KS} \\
				&$n$& Bias & SD & SE & CP & Bias & SD & SE & CP & Bias & SD & SE & CP & Bias & SD & SE & CP \\
				\hline
				\multicolumn{4}{l}{$E\{Z(t)\}=2$}\\
				$\beta$ & $100$ & $0.002$ & $0.099$ & $0.091$ & $92$ & $-0.005$ & $0.130$ & $0.116$ & $91$ & $-0.005$ & $0.130$ & $0.116$ & $91$ & $0.001$ & $0.124$ & $0.113$ & $93$ \\ 
				& $400$ & $0.0004$ & $0.049$ & $0.047$ & $94$ & $-0.008$ & $0.064$ & $0.059$ & $93$ & $-0.008$ & $0.064$ & $0.059$ & $93$  & $0.001$ & $0.082$ & $0.081$ & $94$ \\ 
				& $900$ & $-0.002$ & $0.032$ & $0.031$ & $94$ & $-0.008$ & $0.040$ & $0.040$ & $94$ & $-0.008$ & $0.040$ & $0.040$ & $94$  & $-0.001$ & $0.072$ & $0.069$ & $93$ \\ 
				\hline
				$\gamma$ & $100$ & $0.121$ & $0.102$ & $0.095$ & $74$ & $0.128$ & $0.102$ & $0.095$ & $72$ & $0.023$ & $0.129$ & $0.122$ & $91$ & $0.016$ & $0.129$ & $0.116$ & $91$ \\ 
				& $400$ & $0.118$ & $0.051$ & $0.049$ & $34$ & $0.120$ & $0.051$ & $0.049$ & $33$ & $0.006$ & $0.084$ & $0.085$ & $94$ & $0.004$ & $0.084$ & $0.083$ & $95$ \\ 
				& $900$ & $0.118$ & $0.032$ & $0.033$ & $5$ & $0.119$ & $0.032$ & $0.033$ & $5$ & $0.0002$ & $0.071$ & $0.072$ & $95$ & $-0.0005$ & $0.071$ & $0.072$ & $95$ \\ 
				\hline
				$\alpha$ & $100$ & $-0.245$ & $0.242$ & $0.226$ & $80$ & $-0.254$ & $0.249$ & $0.217$ & $76$ & $-0.043$ & $0.308$ & $0.286$ & $91$ & $-0.032$ & $0.305$ & $0.273$ & $91$ \\ 
				& $400$ & $-0.240$ & $0.121$ & $0.116$ & $47$ & $-0.237$ & $0.125$ & $0.112$ & $46$ & $-0.006$ & $0.193$ & $0.194$ & $94$ & $-0.008$ & $0.197$ & $0.194$ & $96$ \\ 
				& $900$ & $-0.236$ & $0.079$ & $0.078$ & $14$ & $-0.234$ & $0.081$ & $0.075$ & $13$ & $0.002$ & $0.161$ & $0.163$ & $94$ & $-0.001$ & $0.166$ & $0.167$ & $95$ \\ 
				\hline
				\multicolumn{4}{l}{$E\{Z(t)\}=0.5+t$}\\
				$\beta$ & $100$ & $-0.011$ & $0.093$ & $0.091$ & $94$ & $-0.003$ & $0.120$ & $0.114$ & $92$ & $-0.003$ & $0.120$ & $0.114$ & $92$ & $-0.002$ & $0.122$ & $0.113$ & $92$ \\ 
				& $400$ & $-0.011$ & $0.046$ & $0.047$ & $95$ & $-0.001$ & $0.060$ & $0.059$ & $94$ & $-0.001$ & $0.060$ & $0.059$ & $94$ & $-0.001$ & $0.085$ & $0.082$ & $93$ \\ 
				& $900$ & $-0.010$ & $0.031$ & $0.032$ & $94$ & $-0.001$ & $0.040$ & $0.039$ & $94$ & $-0.001$ & $0.040$ & $0.039$ & $94$ & $-0.003$ & $0.069$ & $0.070$ & $95$ \\ 
				\hline
				$\gamma$ & $100$ & $0.114$ & $0.095$ & $0.092$ & $74$ & $0.120$ & $0.095$ & $0.092$ & $72$ & $0.008$ & $0.120$ & $0.117$ & $91$ & $0.002$ & $0.119$ & $0.112$ & $91$ \\ 
				& $400$ & $0.119$ & $0.047$ & $0.047$ & $30$ & $0.120$ & $0.047$ & $0.047$ & $28$ & $0.004$ & $0.082$ & $0.081$ & $94$ & $0.002$ & $0.082$ & $0.080$ & $95$ \\ 
				& $900$ & $0.119$ & $0.031$ & $0.032$ & $2$ & $0.119$ & $0.031$ & $0.031$ & $3$ & $0.003$ & $0.068$ & $0.069$ & $94$ & $0.003$ & $0.069$ & $0.069$ & $95$ \\ 
				\hline
				$\alpha$ & $100$ & $-0.230$ & $0.150$ & $0.148$ & $64$ & $-0.243$ & $0.156$ & $0.136$ & $56$ & $-0.010$ & $0.186$ & $0.187$ & $94$ & $-0.004$ & $0.189$ & $0.180$ & $93$ \\
				& $400$ & $-0.236$ & $0.078$ & $0.076$ & $13$ & $-0.244$ & $0.083$ & $0.069$ & $9$ & $-0.001$ & $0.126$ & $0.124$ & $94$ & $0.000$ & $0.131$ & $0.128$ & $95$ \\ 
				& $900$ & $-0.234$ & $0.051$ & $0.051$ & $0$ & $-0.241$ & $0.054$ & $0.047$ & $0$ & $0.000$ & $0.102$ & $0.104$ & $94$ & $0.002$ & $0.107$ & $0.109$ & $96$ \\ 
				\hline
				\multicolumn{4}{l}{$E\{Z(t)\}=0.5+t^2$}\\
				$\beta$ & $100$ & $-0.021$ & $0.094$ & $0.091$ & $93$ & $-0.003$ & $0.125$ & $0.115$ & $93$ & $-0.003$ & $0.125$ & $0.115$ & $93$ & $-0.005$ & $0.125$ & $0.112$ & $92$ \\ 
				& $400$ & $-0.021$ & $0.047$ & $0.047$ & $93$ & $-0.001$ & $0.059$ & $0.059$ & $95$ & $-0.001$ & $0.059$ & $0.059$ & $95$ & $-0.001$ & $0.083$ & $0.081$ & $94$ \\ 
				& $900$ & $-0.018$ & $0.032$ & $0.032$ & $91$ & $0.003$ & $0.040$ & $0.040$ & $95$ & $0.003$ & $0.040$ & $0.040$ & $95$ & $0.003$ & $0.071$ & $0.069$ & $95$ \\ 
				\hline
				$\gamma$ & $100$ & $0.109$ & $0.091$ & $0.091$ & $77$ & $0.115$ & $0.090$ & $0.091$ & $75$ & $0.018$ & $0.119$ & $0.114$ & $92$ & $0.011$ & $0.120$ & $0.110$ & $93$ \\ 
				& $400$ & $0.112$ & $0.048$ & $0.047$ & $33$ & $0.113$ & $0.048$ & $0.047$ & $33$ & $0.009$ & $0.085$ & $0.080$ & $92$ & $0.008$ & $0.085$ & $0.079$ & $92$ \\ 
				& $900$ & $0.109$ & $0.033$ & $0.032$ & $7$ & $0.109$ & $0.033$ & $0.032$ & $7$ & $0.000$ & $0.070$ & $0.068$ & $94$ & $-0.001$ & $0.070$ & $0.068$ & $94$ \\ 
				\hline
				$\alpha$ & $100$ & $-0.202$ & $0.143$ & $0.140$ & $68$ & $-0.219$ & $0.153$ & $0.126$ & $56$ & $-0.013$ & $0.185$ & $0.175$ & $91$ & $-0.006$ & $0.189$ & $0.169$ & $92$ \\  
				& $400$ & $-0.203$ & $0.072$ & $0.072$ & $20$ & $-0.218$ & $0.075$ & $0.064$ & $12$ & $-0.008$ & $0.119$ & $0.114$ & $93$ & $-0.007$ & $0.123$ & $0.119$ & $93$ \\ 
				& $900$ & $-0.201$ & $0.049$ & $0.048$ & $1$ & $-0.216$ & $0.051$ & $0.043$ & $0$ & $-0.001$ & $0.098$ & $0.095$ & $93$ & $-0.000$ & $0.104$ & $0.102$ & $94$ \\ 
				\hline
				\multicolumn{4}{l}{$E\{Z(t)\}=2\sin(2\pi t)$}\\
				$\beta$ & $100$ & $0.087$ & $0.111$ & $0.111$ & $86$ & $0.010$ & $0.121$ & $0.116$ & $94$ & $0.010$ & $0.121$ & $0.116$ & $94$ & $0.004$ & $0.123$ & $0.114$ & $93$ \\ 
				& $400$ & $0.082$ & $0.059$ & $0.056$ & $69$ & $0.008$ & $0.061$ & $0.059$ & $94$ & $0.008$ & $0.061$ & $0.059$ & $94$ & $-0.001$ & $0.086$ & $0.082$ & $93$ \\ 
				& $900$ & $0.084$ & $0.038$ & $0.038$ & $40$ & $0.009$ & $0.043$ & $0.040$ & $92$ & $0.009$ & $0.043$ & $0.040$ & $93$ & $0.001$ & $0.073$ & $0.070$ & $94$ \\ 
				\hline
				$\gamma$ & $100$ & $0.257$ & $0.059$ & $0.056$ & $1$ & $0.255$ & $0.059$ & $0.055$ & $1$ & $0.009$ & $0.064$ & $0.061$ & $91$ & $0.006$ & $0.065$ & $0.061$ & $93$ \\ 
				& $400$ & $0.257$ & $0.028$ & $0.029$ & $0$ & $0.253$ & $0.028$ & $0.028$ & $0$ & $0.003$ & $0.046$ & $0.046$ & $93$ & $0.002$ & $0.047$ & $0.046$ & $95$ \\ 
				& $900$ & $0.257$ & $0.019$ & $0.019$ & $0$ & $0.253$ & $0.019$ & $0.019$ & $0$ & $0.001$ & $0.041$ & $0.040$ & $93$ & $-0.0003$ & $0.041$ & $0.040$ & $94$ \\ 
				\hline
				$\alpha$ & $100$ & $0.279$ & $0.154$ & $0.143$ & $50$ & $0.336$ & $0.157$ & $0.119$ & $25$ & $-0.011$ & $0.157$ & $0.150$ & $92$ & $-0.006$ & $0.158$ & $0.147$ & $93$ \\ 
				& $400$ & $0.282$ & $0.077$ & $0.073$ & $3$ & $0.336$ & $0.078$ & $0.060$ & $0$ & $-0.008$ & $0.097$ & $0.095$ & $93$ & $-0.001$ & $0.104$ & $0.102$ & $94$ \\ 
				& $900$ & $0.282$ & $0.049$ & $0.049$ & $0$ & $0.338$ & $0.050$ & $0.040$ & $0$ & $-0.005$ & $0.080$ & $0.078$ & $93$ & $0.0003$ & $0.088$ & $0.087$ & $94$ \\ 
				\hline
		\end{tabular} }
	\end{center}
	\footnotesize{Note: ``Bias" is the empirical bias, ``SD'' is the sample standard deviation, ``SE'' is the average of the standard error estimates, ``CP''$/100$ represents the coverage probability of the $95\%$ confidence interval for $\hat{\beta},\hat\gamma$, and $\hat\alpha$, respectively.}
\end{table}

For comparison, we implement three alternative estimation procedures. 
\begin{enumerate}
	\item Apply last value carried forward (LVCF) to estimate $\beta$ and $\gamma$ simultaneously. In longitudinal studies, a na\"ive approach to analyzing asynchronous longitudinal data is the last value carried forward method. If data at a certain time point are missing, the observation at the most recent time point in the past is used in analysis for synchronous data. This method is referred to as LVCF in Table \ref{Ch3: table all}. Specifically, for $i$th subject, at time $t_{ij},$ if the covariate process $Z_i(\cdot)$ does not have any observation, then the most recently observed $Z_{i}(s)$ is used, where $s = \mbox{max}\{x \le t_{ij}, x\in \{s_{i1}, \ldots, s_{iM_i} \} \}.$ After this imputation, we proceed with the usual least square estimation procedure.
	\item Apply the two-step approach, but use LVCF to estimate $\gamma$ in the second step. Specifically, for $i$th subject, in the first stage, obtain the longitudinal residual, $\hat{\omega}_i(t_{ij}) = Y_i(t_{ij}) -X_i(t_{ij})^T\hat{\beta}.$ In the second stage, regress $\hat{\omega}_{i}(t_{ij})$ with $Z_i(s),$ where $s= \mbox{max}\{x \le t_{ij}, x\in \{s_{i1}, \ldots, s_{iM_i} \} \}.$ This method is referred to as Centering + LVCF in Table \ref{Ch3: table all}.
	\item 
	Solve the estimating equation (\ref{ee2-1}) to obtain estimators of $\beta$ and $\gamma$ simultaneously. The estimation of $\gamma$ is similar to the proposed two-step approach, yet the estimation of $\beta$ is less efficient than the proposed two-step approach. This method is referred as KS in Table \ref{Ch3: table all}.
\end{enumerate}

Automatic bandwidth selection procedure proposed in Section \ref{Ch3: subsec bdwd} is used where the bandwidths are selected in the range $(n^{-0.8}, n^{-0.6}).$ 

We summarize simulation results in Table \ref{Ch3: table all}. For estimation of $\beta,$ when $E\{Z(t) \}=2,$ all methods perform satisfactorily as assumptions for LVCF are satisfied in this case. When $E\{Z(t)\}$ is non-constant, the performance of LVCF and Centering + LVCF deteriorates and Centering + KS and KS both produce valid results. As our theory predicts, Centering + KS is more efficient than KS for estimation of $\beta$, which is reflected on the smaller variance. For estimation of $\gamma$ and $\alpha,$ both LVCF and Centering + LVCF are biased while Centering + KS and KS produce similar results. 

\section{Application to the ADNI data}
In this section, we illustrate the proposed method of analyzing mixed synchronous and asynchronous longitudinally observed functional data on a study of the Alzheimer's disease. In the dataset, $256$ subjects were followed for $5$ years. The dataset is collected from ADNI GO and ADNI2 in the ADNI study. Among many goals of the ADNI study, we are  interested in clinical, functional neuroimaging and structural variables that affect the progression of mild cognitive impairment and early Alzheimer's disease. The response variable MMSE ranges from $0$ to $30$ measuring global cognitive performance, where larger values means better cognitive state. It is examined from $1$ to $7$ time points. Baseline covariates include age, years of education, whether the person has mild cognitive impairment (MCI), whether the person has early Alzheimer's disease (AD), the number of APOE4 gene copies and the log hazard function of fractional anisotropy (FA) at grid point $0.65$. The FA is one of the most popular diffusion-weighted imaging measures that reflects fiber density and myelination in white matter, which is observed at $1$ to $8$ time points. Details of the data processing are given in \cite{lizhu2022}. The measurement time points of the log hazard functions of FA and the MMSE scores are different between and within subjects while baseline measurements are in alignment with the MMSE score, giving rise to mixed synchronous and asynchronous longitudinal covariates. 

We use model (\ref{true-model}) to fit the data. Our modeling assumes that the asynchronous longitudinal covariate FA is not correlated with baseline covariates. To better understand this, we regress FA against age, years of education, MCI, AD and APOE4(1) and APOE(2), first in six univariate regression models and then in one multiple regression model. The results are summarized in Table \ref{Ch3: table cd4}. The $p$-values are computed based on two-sided test. We observe that none of the six baseline covariates are statistical significantly associated with FA. Consequently, we include them all in our model. 
\begin{table}[!ht]
	\caption{Regression of FA on six covariates}\label{Ch3: table cd4}
	\begin{center}
		\begin{tabular}{lrrrrrr}
			\hline
			& Age & Education & MCI & AD&APOE4(1)& APOE4(2)\\
			& &&\multicolumn{2}{c}{Fit separately} &&\\ Estimate & 0.006& 0.003 & 0.057& $-$0.020 & $-$0.006 & $-$0.070\\
			$p$-value & 0.067&	0.634&	0.120& 0.601& 0.875&0.286\\
			\hline
			&&&\multicolumn{2}{c}{All in one model} &&\\
			Estimate & 0.005 & 0.004 & 0.075 & 0.026 & $-$0.013 & $-$0.067\\
			$p$-value & 0.098 & 0.520 & 0.101 & 0.623 & 0.766 & 0.326\\
			\hline
		\end{tabular}
	\end{center}
\end{table}

We fit model (\ref{true-model}) with bandwidth chosen in the range $(2(Q_3-Q_1)n^{-0.7}, \newline 2(Q_3-Q_1)n^{-0.6}),$ where $Q_3$ and $Q_1$ are third and first quantiles of the combined observation times of MMSE and FA, $n=256$ is number of patients after eliminating missing data \citep{little2014}. We implemented five methods: Na\"ive, LVCF, Centering + LVCF, Centering + KS and KS. The na\"ive method ignores FA in fitting the linear regression model. In the two-step approach, regression coefficients of age, education, MCI, AD, APOE4(1) and APOE4(2) are estimated in the first step using the centering approach, and the regression coefficient of FA is estimated in the second step using either LVCF (Centering + LVCF) or kernel smoothing (Centering + KS). LVCF and KS estimate all regression coefficients simultaneously. We normalize continuous variables including MMSE, age, education and HA. Analysis results are summarized in Table \ref{Ch3: table hiv}.
\begin{table}[!ht]
	\caption{Analysis of dataset from ADNI} 
	\label{Ch3: table hiv}\small   \setlength{\tabcolsep}{3pt}
	\begin{center}\small
		\scalebox{0.8}{\begin{tabular}{lrrrrrrrrrr} 
				\hline
				& \multicolumn{2}{c}{Na\"ive} & \multicolumn{2}{c}{LVCF} & \multicolumn{2}{c}{Centering+LVCF} &  \multicolumn{2}{c}{Centering+KS} & \multicolumn{2}{c}{\cite{cao15}}\\
				& Estimate & $p$-value & Estimate & $p$-value & Estimate  & $p$-value & Estimate  & $p$-value & Estimate  & $p$-value \\
				\hline
				Age& $-$0.199 & 0.258 & $-$0.241 & 0.199 & $-$0.280 & 0.728 & $-$0.280 & 0.728 &$-$0.176 & 0.320\\
				Edu & 0.117 & 0.001 & 0.111 & 0.011 &0.116 & 0.000&0.116 & 0.000&0.106 & 0.001\\
				MCI & $-$0.393 & 0.000 & $-$0.427 & 0.000 & $-$0.422 & 0.000 & $-$0.422 & 0.000 &$-$0.352 & 0.000\\
				AD& $-$1.779 & 0.000 & $-$1.764 & 0.000 & $-$1.858 & 0.000 & $-$1.858 & 0.000& $-$1.742 & 0.000\\
				AP4(1)&$-$0.329 & 0.000 & $-$0.328 & 0.000 & $-$0.305 & 0.000&$-$0.305 & 0.000&$-$0.273 & 0.000\\
				AP4(2) & $-$0.391 & 0.022 & $-$0.366 & 0.036&$-$0.415 & 0.017 & $-$0.415 & 0.017 & $-$0.296 & 0.039\\
				FA &\multicolumn{2}{c}{(omitted)}& 0.047 & 0.225 & 0.044 & 0.248 & 0.061& 0.047& 0.058 & 0.062\\
				\hline
		\end{tabular}}
	\end{center}
	\footnotesize{Note: ``AP4(1)'' is APOE4(1), and ``AP4(2)''  is APOE4(2).}
\end{table} 

There is no statistically significant association between age and MMSE across all methods. As education, MCI, AD, APOE(1) and APOE(2) are baseline covariates, their parameter estimates are quite similar across different methods and all show statistical significance. The estimation results show that MCI at baseline, AD at baseline, APOE4(1) and APOE4(2) have significant negative effect on MMSE, whereas education plays a positive role, which has been verified in the literature \citep{bekris2010}. For misaligned FA, LVCF and Centering + LVCF do not show statistical significance. However, the newly proposed Centering + KS shows that it has a statistically significant positive effect. This finding is consistent with the existing literature that lower FA values, which means less white matter, are associated with lower MMSE scores \citep{kristensen2019}.

\section{Concluding remarks}

In this paper, we propose valid statistical approaches for analysis of longitudinal data with omitted longitudinal covariates. Furthermore, to deal with mixed synchronous and asynchronous longitudinal covariates, we propose a two-step approach for analysis. In the first step, a partial linear model approach or a centering approach is used to estimate regression coefficient of the synchronous longitudinal covariate. In the second step, we regress asynchronous longitudinal covariate with longitudinal residual from the first step through kernel weighting to obtain regression coefficient estimation of the asynchronous longitudinal covariate. We obtain parametric root $n$ rate of convergence for regression coefficient estimation of the synchronous longitudinal covariate and $n^{2/5}$ rate of convergence for regression coefficient estimation of the asynchronous longitudinal covariate. In terms of implementation, the centering approach is computationally faster while the partial linear model approach can suggest possible forms of the omitted longitudinal covariate. We require that asynchronous longitudinal covariate is uncorrelated with synchronous longitudinal covariate. If this assumption is violated, we suggest to model the synchronous and asynchronous longitudinal covariates jointly to get unbiased regression coefficient estimation. 

We use working independence covariance matrix for analysis. \cite{pepe1994} pointed out that working independence covariance matrix is a safe choice to get unbiased estimation of the regression coefficient in longitudinal data analysis with time-varying covariates. A carefully chosen working covariance matrix can lead to efficiency gain compared to the simple independent covariance matrix. A fully efficient estimator requires a correctly specified working covariance matrix, which is difficult in practice. This is similar in spirit to the use of ordinary least squares in the presence of heteroscedasticity in linear models.

\section*{Acknowledgements}
We thank Ting Li for help with data acquisition, Peng Ding for helpful discussions and Congmin Liu for help in the data analysis. Data used in preparation of this article were obtained from the Alzheimers Disease Neuroimaging Initiative (ADNI) database (adni.loni.usc.edu). As such, the investigators within the ADNI contributed to the design and implementation of ADNI and/or provided data but did not participate in analysis or writing of this report. A complete listing of ADNI investigators can be found at: \url{
	http://adni.loni.usc.edu/wp-content/uploads/how_to_apply/ADNI_Acknowledgement_List.pdf.}

\newpage
\appendix
In Section A, we provide detailed proofs of Theorems $1$ to $5$. Our main tools are empirical processes \citep{vw95}. In Section B, we list additional simulations. 

\appendix
\section{Proofs of main results}
\subsection{Proof of Theorem 1}

\begin{theorem}\label{consis}
	Under (C1) and (C2), estimation of $\beta$ in (2.1) under the mis-specified model (2.2) is unbiased.
\end{theorem}
\begin{proof}
	We re-parameterize $(2.1)$ as follows. 
	\begin{eqnarray*}
		Y(t) &=& \alpha + E\{Z(t)\}^T\gamma + X(t)^T\beta + [Z(t) - E\{Z(t) \}]^T\gamma + \epsilon(t)\\
		&=& \alpha^\diamond + X(t)^T\beta^\diamond + \epsilon^\diamond (t),
	\end{eqnarray*}
	where $\alpha^\diamond = \alpha+ E\{ Z(t)\}^T\gamma,$ $\beta^\diamond = \beta$ and $\epsilon^\diamond (t) = [Z(t) - E\{Z(t) \}]^T\gamma + \epsilon(t).$ Under (C1), $\alpha^{\diamond}$ is constant\textcolor{black}{, and $E[\alpha+E\{Z_{ij}\}-\alpha_0]=E\left\{\left[Z_{ij}-E\{Z_{ij}\}\right]^T\gamma\right\}=0.$} Under (C2), the new re-parametrized error term $\epsilon^\diamond(t)$ is a mean $0$ process and uncorrelated with $X(t).$ \textcolor{black}{ Thus, we have
		\begin{align*}
			&I=E\left[\Big(\sum_{i=1}^{n}\sum_{j=1}^{M_i}X_{ij}X_{ij}^T \Big)^{-1} X_{ij}\left\{\alpha+E(Z_{ij})^T \gamma-\alpha^\diamond\right\}\right]&\\
			&=E\left[\Big(\sum_{i=1}^{n}\sum_{j=1}^{M_i}X_{ij}X_{ij}^T \Big)^{-1} X_{ij}\right]E\left[\left\{\alpha+E(Z_{ij})^T \gamma-\alpha^\diamond\right\}\right]=0,
		\end{align*}
		\begin{align*}
			&II=E\left[\Big(\sum_{i=1}^{n}\sum_{j=1}^{M_i}X_{ij}X_{ij}^T \Big)^{-1} X_{ij} \left\{Z_{ij}-E(Z_{ij})\right\}^T\gamma\right]&\\
			&=E\left[\Big(\sum_{i=1}^{n}\sum_{j=1}^{M_i}X_{ij}X_{ij}^T \Big)^{-1} X_{ij}\right]E\left[\left\{Z_{ij}-E(Z_{ij})\right\}^T\gamma\right]=0. \quad i=1,\ldots,n;j=1,\ldots,M_i.
		\end{align*}
	}
	Consequently, the mis-specified model (2.2) can be used to correctly estimate the regression coefficient $\beta$ in (2.1).
\end{proof}
\subsection{Proof of Theorem 2}
\begin{theorem}\label{thm1}
	Under (C2)-(C6), the asymptotic distribution of $\hat{\beta}_p$ satisifies
	\begin{equation}\label{asym-plm}
		\sqrt{n}(\hat\beta_p-\beta_0)\overset{d}{\rightarrow} N(0,A_{\textcolor{black}{\beta_0}}^{-1}\Sigma_{\textcolor{black}{\beta_0}} A_{\textcolor{black}{\beta_0}}^{-1}),
	\end{equation} 
	where $\beta_0$ is the true value of $\beta$ in (2.1),
	\begin{eqnarray*}
		A_{\textcolor{black}{\beta_0}}&=&E\int \{\tilde X(t)\tilde X(t)^T\}\lambda(t)dt,\quad \mbox{and}\\
		\Sigma_{\textcolor{black}{\beta_0}}&=&E \Big\{\int \tilde{X}(t) \{\tilde{Z}(t)^T\gamma_0 + \epsilon(t) \}dN(t)\Big\}^{\otimes2}.
	\end{eqnarray*}
\end{theorem}
\begin{proof}
	From \cite{fan92}, for each given $\beta,$ the local linear estimator $\hat{\alpha}(t;\beta)$ in (2.5) is a consistent estimate of the function 
	\begin{align}\label{intercept}
		\alpha(t; \beta) = E\{Y(t) - X(t)^T\beta\} = \alpha_0 + E\{Z(t)\}^T\gamma_0 - E\{X(t)\}^T (\beta - \beta_0) 
	\end{align}
	and $\hat{\alpha}(t;\beta) - \alpha(t;\beta) = O_p\{h^2 + (nh)^{-1/2}\}.$ Let $l_n(\beta)$ denote the quadratic loss,
	\begin{eqnarray*}
		l_n(\beta) = n^{-1}\sum_{i=1}^{n}\sum_{j=1}^{M_i}\{Y_i(t_{ij}) - \hat{\alpha} (t_{ij}; \beta) - X_i(t_{ij})^T\beta \}^2.
	\end{eqnarray*}
	Then $\hat{\beta}_p$ minimizes the convex function $l_n(\beta).$ Decompose 
	$$
	l_n(\beta) = I_{n,1}(\beta) + I_{n,2}(\beta) + I_{n,3}(\beta),
	$$
	where
	$$
	I_{n,1}(\beta) = n^{-1}\sum_{i=1}^{n}\sum_{j=1}^{M_i}\{Y_i(t_{ij}) - \alpha(t_{ij}; \beta) - X_i(t_{ij})^T\beta \}^2,
	$$
	$$
	I_{n,2}(\beta) = 2 n^{-1}\sum_{i=1}^{n}\sum_{j=1}^{M_i}\{Y_i(t_{ij}) - \alpha(t_{ij};\beta) -X_i(t_{ij})^T\beta\}\{ \alpha(t_{ij};\beta) -\hat{\alpha}(t_{ij};\beta)\},
	$$
	and
	$$
	I_{n,3}(\beta) = n^{-1}\sum_{i=1}^{n}\sum_{j=1}^{M_i}\{\alpha(t_{ij};\beta) - \hat{\alpha}(t_{ij};\beta)\}^2.
	$$
	Following \cite{fan92},
	$$
	I_{n,2}(\beta) = O_p\{ I_{n,3}(\beta)\} = O_p(h^4+(nh)^{-1}).
	$$ 
	Using the model $Y(t) = \alpha_0 + Z(t)^T \gamma_0 + X(t)^T \beta_0 + \epsilon(t)$ and (\ref{intercept}), we have
	\begin{align*}
		I_{n,1}(\beta) &= n^{-1}\sum_{i=1}^{n}\int \Big( [ Z_i(t)-E\{Z_i(t)\} ]^T \gamma_0 + \epsilon(t) - [ X_i(t) - E\{X_i(t)\} ]^T (\beta - \beta_0)\Big)^2\\
		& \times dN_i(t)\\
		&=n^{-1}\sum_{i=1}^{n}\Big( [Z_i(t) - E\{Z_i(t)\} ]^T\gamma_0 + \epsilon(t)\Big)^2dN_i(t) \\
		&+ (\beta-\beta_0)^Tn^{-1}\sum_{i=1}^{n}\int [ X_i(t)-E\{X_i(t)\} ] [ X_i(t)-E\{X_i(t)\} ]^TdN_i(t)(\beta - \beta_0)\\
		&-2(\beta - \beta_0)^Tn^{-1}\sum_{i=1}^{n}\int [ X_i(t) - E\{X_i(t)\} ]\Big([Z_i(t) - E\{Z_i(t)\}]^T\gamma_0 + \epsilon(t)\Big)dN_i(t).
	\end{align*}
	The minimizer of this quadratic function is given by 
	\begin{equation*}
		\beta_0 + \Big\{ n^{-1}\sum_{i=1}^{n}\int \tilde{X}_i(t)\tilde{X}_i(t)^TdN_i(t) \Big\}^{-1} \Big\{n^{-1}\sum_{i=1}^{n}\int \tilde{X}_i(t)[\tilde{Z}_i(t)^T\gamma_0 + \epsilon(t)] dN_i(t)\Big\},
	\end{equation*}
	where $\tilde{X}_i(t) = X_i(t) - E\{X_i(t)\}$ and $\tilde{Z}_i(t) = Z_i(t) - E\{Z_i(t)\}.$
	Consequently
	\begin{align*}
		\hat{\beta}_p &= \beta_0 + \Big\{ n^{-1}\sum_{i=1}^{n}\int \tilde{X}_i(t)\tilde{X}_i(t)^TdN_i(t) \Big\}^{-1} \Big\{n^{-1}\sum_{i=1}^{n}\int \tilde{X}_i(t)[\tilde{Z}_i(t)^T\gamma_0 + \epsilon(t)] dN_i(t)\Big\}\\
		& + o_p(1).
	\end{align*}
	\textcolor{black}{Since $\int \tilde{X}_i(t)\tilde{X}_i(t)^TdN_i(t), i=1,\ldots,n.$ are i.i.d. random variables with mean $A_{\beta_0}$}, by the weak law of large numbers and (\ref{intercept}),
	$$
	n^{-1}\sum_{i=1}^{n}\int \tilde{X}_i(t)\tilde{X}_i(t)^TdN_i(t) \overset{p}{\rightarrow}  A_{\textcolor{black}{\beta_0}} . 
	$$
	\textcolor{black}{Similarly, $\int \tilde{X}_i(t)[\tilde{Z}_i(t)^T\gamma_0 + \epsilon(t)] dN_i(t), i=1,\ldots,n.$ are i.i.d. random variables with  $$E\left[\int \tilde{X}_i(t)[\tilde{Z}_i(t)^T\gamma_0 + \epsilon(t)] dN_i(t)\right]=\int E\left[\tilde{X}_i(t)\right]E\left[\tilde{Z}_i(t)^T\gamma_0 + \epsilon(t) dN_i(t)\right]=0,$$ and 
		\begin{align*}
			& Var\left[\int \tilde{X}_i(t)[\tilde{Z}_i(t)^T\gamma_0 + \epsilon(t)dN_i(t)\right] =E\left[\left(\int \tilde{X}_i(t)[\tilde{Z}_i(t)^T\gamma_0 + \epsilon(t)dN_i(t)\right)^2\right] <\infty.
	\end{align*}} By the central limit theorem,
	$$
	n^{-1}\sum_{i=1}^{n}\int \tilde{X}_i(t)[\tilde{Z}_i(t)^T\gamma_0 + \epsilon(t)] dN_i(t) \overset{d}{\rightarrow} N(0, \Sigma_\beta),
	$$
	where by (C2),
	\begin{eqnarray*}
		E\int \tilde{X}(t) [\tilde{Z}(t)^T \gamma_0 + \epsilon(t)]dN(t) &=& E\Big[E\int \tilde{X}(t)\{\tilde{Z}(t)^T\gamma_0 + \epsilon(t) \}dN(t)\mid \tilde{X}(t), \tilde{Z}(t) \Big]\\
		&=& E\Big[ \int \tilde{X}(t)\tilde{Z}(t)^T\gamma_0 \lambda(t)dt\Big]\\
		&=& \int  \lambda(t) E\{ \tilde{X}(t) \tilde{Z}(t)^T \} \gamma_0 dt\\
		&=&0_{p \times 1}.
	\end{eqnarray*}
	Therefore, by Slutsky theorem,
	\begin{equation*}
		\sqrt{n}(\hat{\beta}_p - \beta_0) \overset{d}{\rightarrow} N(0, A^{-1}_{\textcolor{black}{\beta_0}} \Sigma_{\textcolor{black}{\beta_0}} A^{-1}_{\textcolor{black}{\beta_0}}).
	\end{equation*}
\end{proof}

\subsection{Proof of Theorem 3}
\begin{theorem} \label{thm2}
	Under (C2)-(C7), the asymptotic distribution of $\hat{\beta}_c$ satisfies
	\begin{equation*}
		\sqrt{n}(\hat{\beta}_c - \beta_0) \overset{d}{\rightarrow} N(0, A_{\textcolor{black}{\beta_0}}^{-1}\Sigma_{\textcolor{black}{\beta_0}} A_{\textcolor{black}{\beta_0}}^{-1}),
	\end{equation*}
	where $A_\beta$ and $\Sigma_\beta$ are the same as those in Theorem $2$.
\end{theorem}
\begin{proof}
	As in \cite{fan92}, under (C5) and (C6), it can be shown that $\hat{m}_Y(t) - m_Y(t) = O_p(h^2 + (nh)^{-1/2})$ and $\hat{m}_X(t) - m_X(t) = O_p(h^2 + (nh)^{-1/2}).$ Then the estimating equation (2.11) becomes
	\begin{eqnarray*}
		U(\beta) &= &n^{-1}\sum_{i=1}^{n}\int \tilde{X}_i(t)\{\tilde{Y}_i(t) - \tilde{X}_i(t)^T\beta \}dN_i(t) + o_p(1)\\
		&=& U^{*}(\beta) + o_p(1).
	\end{eqnarray*}
	By the weak law of large numbers, 
	$$
	U^*(\beta) \overset{p}{\rightarrow} u^*(\beta),
	$$
	where 
	$$
	u^*(\beta) = E\int \{ \tilde{X}(t) \tilde{X}(t)^T \}\lambda(t)dt(\beta_0-\beta).
	$$
	By (C4), $\beta_0$ is the unique solution to $u^*(\beta)=0.$ As $\hat{\beta}_c$ solves the estimating equation (2.11), it follows from the convexity lemma \citep{ag1982} that $\hat{\beta}_c \overset{p}{\rightarrow}\beta_0.$
	
	We next show the asymptotic normality of $\hat{\beta}_c.$ Let $\mathcal P_n$ and $\mathcal P$ denote the empirical measure and the probability measure, respectively. Then
	\begin{align}
		\sqrt{n}U^*(\beta)&=\sqrt{n}(\mathcal P_n-\mathcal P)\left[\int \tilde X(t)\{\tilde Y(t)-\tilde X(t)^T\beta \}d N(t)\right]  \notag\\
		&+\sqrt{n}E\left[\int \tilde X(t)\{\tilde Y(t)-\tilde X(t)^T\beta \}d N(t)\right]. \tag{A.2} \label{Ch3: app eq u beta}
	\end{align}
	Consider the class of functions
	\begin{align*}
		\left\{\int \tilde X(t)\{\tilde Y(t)-\tilde X(t)^T\beta \}d N(t): |\beta-\beta_0|<\epsilon\right\}
	\end{align*}
	for a given constant $\epsilon.$ Note that the functions in this class are differentiable and their corresponding first-order derivatives are bounded according to condition (C4),
	therefore the functions in this class are Lipschitz continuous in $\beta$ and
	the Lipschitz constant is uniformly bounded by $\|\int \tilde X(t) \tilde X(t)^T dN(t)\|_{\infty}.$ It can be shown that this class is a P-Donsker class \citep{vw95}. As a result, we obtain that for $|\beta-\beta_0| <\epsilon$, the first term in (\ref{Ch3: app eq u beta}) is equal to 
	\begin{eqnarray*}
		&& \sqrt{n}(\mathcal P_n-\mathcal P)\left[\int \tilde X(t)\{\tilde Y(t)-\tilde X(t)^T\beta_0 \}d N(t)\right]+o_p(1)\\
		&=&\sqrt{n}[U^*(\beta_0)-E\{U^*(\beta_0)\}]+o_p(1).
	\end{eqnarray*}
	The second term in (\ref{Ch3: app eq u beta}) is equal to, 
	\begin{eqnarray*}
		&&\sqrt{n}E\left[\int \tilde X(t)\{\tilde Y(t)-\tilde X(t)^T(\beta-\beta_0)-\tilde X(t)^T\beta_0 \} dN(t)\right]\\
		&=&\sqrt{n}E\left[\int \tilde X(t) E\{\tilde Z(t)^T\gamma_0+\epsilon(t)-\tilde X(t)^T(\beta-\beta_0) \mid \tilde X(t) \} \lambda(t) dt\right]\\
		&=&-\sqrt{n}E\int \{\tilde X(t)\tilde X(t)^T\}\lambda(t)d t(\beta-\beta_0)\\
		&:=&-\sqrt{n}A_{\textcolor{black}{\beta_0}} (\beta-\beta_0).
	\end{eqnarray*}
	Consequently,
	\begin{eqnarray*}
		\sqrt{n}U^*(\beta)&=&\sqrt{n}[U^*(\beta_0)-E\{U^*(\beta_0)\}-A(\beta-\beta_0)]+o_p(1)\\
		&=&\sqrt{n}[U^*(\beta_0)-A(\beta-\beta_0)]+o_p(1),
	\end{eqnarray*}
	as $E\{U^*(\beta_0) \} =0.$
	Let $\phi_i=\int \tilde X_i(t)\{\tilde Y_i(t)-\tilde X_i(t)^T\beta_0 \}d N_i(t)$. Then $\phi_i$ are independent and identically distributed (i.i.d.) and $\sqrt{n}U^*(\beta_0)=n^{-1/2}\sum_{i=1}^n\phi_i$, and 
	\begin{align*}
		E(\phi_i)&=E\left[\int \tilde X(t) E\{\tilde Z(t)^T\gamma_0+\epsilon(t) \mid \tilde X(t) \} \lambda(t) dt\right]=0,\\
		\var(\phi_i)&=E(\phi_i^{\otimes2})=E\left[\iint \{\tilde Y(t)-\tilde X(t)^T\beta_0\}\{\tilde Y(s)-\tilde X(s)^T\beta_0\}\tilde X(t)\tilde X(s)^T dN(t) dN(s) \right]\\
		&=E\left[\iint \{\tilde Z(t)^T\gamma_0+\epsilon(t)\}\{\tilde Z(s)^T\gamma_0+\epsilon(s) \}\tilde X(t)\tilde X(s)^T dN(t) dN(s) \right] \\
		& := \Sigma_{\textcolor{black}{\beta_0}}.
	\end{align*}
	Then by the Central Limit Theorem,
	\begin{align*}
		\sqrt{n}A_\beta (\hat\beta_c-\beta_0) = \sqrt{n}U^*(\beta_0)+o_p(1) \overset{d}{\rightarrow} N(0,\Sigma_{\textcolor{black}{\beta_0}}).
	\end{align*}
	
\end{proof}

\subsection{Proof of Theorem 4}
\begin{theorem} \label{thm3}
	Under (C3)-(C5),(C7)-(C10), the asymptotic distribution of $\hat{\gamma}$ satisfies 
	\begin{equation*}
		\sqrt{nh}(\hat{\gamma} - \gamma_0) \overset{d}{\rightarrow}N(0, A_{\textcolor{black}{\gamma_0}}^{-1}\Sigma_{\textcolor{black}{\gamma_0}} A_{\textcolor{black}{\gamma_0}}^{-1}),
	\end{equation*}
	where
	\begin{align*}
		A_{\textcolor{black}{\gamma_0}}&=\int E\left\{ Z(s)Z(s)^T\right\}\eta(s,s)ds\quad\text{and}\\
		\Sigma_{\textcolor{black}{\gamma_0}}&= \int K(z)^2dz\int E\left\{Z(s)Z(s)^T\right\} \eta(s,s)\sigma^2(s) ds.
	\end{align*}
\end{theorem}
\begin{proof}
	Plugging in $\hat{\beta}_p$ or $\hat{\beta}_c$ into (3.13), by the fact that $\hat{\beta}_p = \beta_0 + O_p(n^{-1/2})$ and $\hat{\beta}_c = \beta_0 + O_p(n^{-1/2}),$ we have
	\begin{align}\label{ee-change}
		&\textcolor{black}{E\{U^f( \gamma) \} } \notag\\
		&= n^{-1}\sum_{i=1}^{n} \textcolor{black}{E \iint K(z) Z_i(s)\{Y_i(s+hz) - Z_i(s)^T\gamma - X_i(s+hz)^T\beta_0 \}\eta(s+hz, s)dsdz} \notag\\
		&+o_p(1),\tag{A.3}
	\end{align}
	where we use change of variable $t = s + hz.$ 
	\textcolor{black}{The right hand side of (\ref{ee-change}) equal to}
	\begin{eqnarray*}
		&&n^{-1}\sum_{i=1}^{n}E\iint K(z)Z_i(s)\{ Z_i(s+hz)^T\gamma_0 - Z_i(s)^T \gamma\} \{\eta(s,s) + \frac{\partial \eta(s,s)}{\partial x} hz \}dsdz +o_p(1)\\ 
		&=& n^{-1}\sum_{i=1}^{n}\int E\{ Z_i(s)Z_i(s)^T\}(\gamma_0 - \gamma) \eta(s,s)ds + O(h^2).
	\end{eqnarray*}
	By weak law of large numbers, we have
	\begin{equation*}
		U^f(\hat{\beta}, \gamma) \overset{p}{\rightarrow} u^f(\gamma),
	\end{equation*}
	where 
	$$
	u^f(\gamma) = \int E\{ Z(s)Z(s)^T\} \eta(s,s)ds (\gamma_0 - \gamma),
	$$
	as $n\rightarrow \infty$ and $h \rightarrow 0.$
	As $\int E\{Z(s)Z(s)^T \}\eta(s,s)ds$ is positive definite, $\gamma_0$ is the unique solution to $u^f(\gamma) =0.$ As $\hat{\gamma}$ solves the estimating equation (3.13), it follows from the convexity lemma \citep{ag1982} that $\hat{\gamma}\overset{p}{\rightarrow} \gamma_0. $
	
	We next show the asymptotic normality of $\hat{\gamma}.$ Let $\mathcal P_n$ and $\mathcal P$ denote the empirical measure and the true probability measure, respectively. We have
	\begin{align}\label{u beta gamma}
		\sqrt{nh}U^f(\hat{\beta}, \gamma)&=\sqrt{nh}(\mathcal P_n-\mathcal P)\left[\iint K_h(t-s) Z(s)\{ Y(t)-Z(s)^T\gamma- X(t)^T \hat{\beta} \} dN(t,s)\right]\notag \\
		&+\sqrt{nh} E\left[\iint K_h(t-s) Z(s)\{ Y(t)-Z(s)^T\gamma- X(t)^T \hat{\beta} \} d N(t,s)\right]. \tag{A.4}
	\end{align}
	For $i$th subject, consider the class of functions
	\begin{align*}
		\Big\{& \sqrt{h} \iint K_h(t-s) Z_i(s)\{ Y_i(t)-Z_i(s)^T\gamma- X_i(t)^T\beta \} dN_i(t,s): \Big.\\
		&\Big. |\beta-\beta_0|<\epsilon_1, |\gamma-\gamma_0|<\epsilon_2 \Big\}
	\end{align*}
	for given constants $\epsilon_1$ and $\epsilon_2$. It can be shown that this class is a P-Donsker class \citep{vw95}. As a result, we obtain that for $|\beta-\beta_0|<M_1n^{-1/2}, |\gamma-\gamma_0| <M_2(nh)^{-1/2}$, where $M_1$ and $M_2$ are some constants, the first term in (\ref{u beta gamma}) is equal to
	\begin{align*}
		&\sqrt{nh}(\mathcal P_n-\mathcal P)\iint K_h(t-s) Z(s)\{ Y(t)-Z(s)^T\gamma_0- X(t)^T\beta_0 \} d N_i(t,s)+o_p(1)\\
		&=\sqrt{nh}[U^f(\beta_0,\gamma_0)- E\{U^f(\beta_0,\gamma_0)\}].
	\end{align*}
	For the second term in (\ref{u beta gamma}), we have 
	\begin{align*}
		&\sqrt{nh} E\left[\iint K_h(t-s) Z(s)\{ Y(t)-Z(s)^{T}(\gamma-\gamma_0)- X(t)^T(\hat{\beta}-\beta_0)-Z(s)^T\gamma_0-X(t)^T\beta_0 \}\right.\\&\left.\times\vphantom{\int} dN(t,s)\right] \\
		&=\sqrt{nh} E\left[\iint K_h(t-s) Z(s)\{ Y(t)-Z(s)^T\gamma_0- X(t)^T\beta_0 \}d N(t,s)\right] \\
		&-\sqrt{nh}\iint K_h(t-s) E\{Z(s)X(t)^T\} \eta(t,s) dtds( \hat{\beta}-\beta_0) \\
		&-\sqrt{nh}\iint K_h(t-s) E\{Z(s)Z(s)^T\}\eta\{t,s)dtds(\gamma-\gamma_0)\\
		&=I_1-I_2-I_3.
	\end{align*}
	For $I_1$, let $\psi_i=\sqrt{h}\iint K_h(t-s) Z_i(s)\{ Y_i(t)-Z_i(s)^T\gamma_0- X_i(t)^T\beta_0 \}d N_i(t,s)$. Then $\psi_i$s are i.i.d. and $\sqrt{nh}U(\beta_0,\gamma_0)=n^{-1/2}\sum_{i=1}^n\psi_i$. Let $t=s+hz$, by conditions (C5), (C9) and (C10), we have
	\begin{align*}
		I_1&=\sqrt{n}E(\psi_i)=\sqrt{nh}\iint K_h(t-s) E\{Z(s)Z(t)^T-Z(s)Z(s)^T\} \eta(t,s) dt ds\gamma_0\\
		&=\sqrt{nh}\iint K(z) E\{Z(s)Z(s+hz)^{T}-Z(s)Z(s)^T\} \eta(s+hz,s) dzds\gamma_0\\
		& = o(1).
	\end{align*}
	Note that by Theorem 1 and Theorem 2, $I_2=O_p(\sqrt{h})$. 
	
	For $I_3$,
	\begin{align*}
		I_3&=\sqrt{nh}\iint K(z) E\{Z(s)Z(s)^T\}\eta(s,s) dzds(\gamma-\gamma_0) + O(n^{1/2}h^{5/2})\\
		&:=\sqrt{nh}A_{\textcolor{black}{\gamma_0}}(\gamma-\gamma_0)+O(n^{1/2}h^{5/2}). 
	\end{align*}
	Consequently,
	\begin{align*}
		\sqrt{nh}U^f(\hat{\beta},\gamma)&=\sqrt{nh}[U^f(\beta_0,\gamma_0) - E\{ U^f(\beta_0, \gamma_0)\}]-\sqrt{nh}A_{\textcolor{black}{\gamma_0}}(\gamma-\gamma_0)+O(n^{1/2}h^{5/2}).\end{align*}
	As $\hat{\gamma}$ solves the estimating equation (3.13), we have 
	$$
	\sqrt{nh} A_{\textcolor{black}{\gamma_0}} (\hat{\gamma} - \gamma_0) = \sqrt{nh}[U^f(\beta_0, \gamma_0) - E\{ U^f(\beta_0, \gamma_0)\}]+ O(n^{1/2}h^{5/2}).
	$$
	The variance of $\sqrt{nh}U^f(\beta_0,\gamma_0)$ can be computed using $\psi_i$ defined earlier
	\begin{align*}
		& \var(\psi_i)\\
		&=E[\var\{\psi_i|X(t),Z(s),N(t,s),t,s \in [0, 1]\}]+\var[E\{\psi_i|X(t),Z(s),N(t,s),t,s \in [0, 1]\}]\\
		&=J_1+J_2,
	\end{align*}
	where
	\begin{align*}
		J_1&=hE\left[\iiiint K_h(t_1-s_1)K_h(t_2-s_2) Z(s_1)Z(s_2)^T\right. \\
		&\left.\times E\{Y(t_1)Y(t_2)|X(t),Z(s),N(t,s),t,s \in [0, 1]\} dN(t_1,s_1) dN(t_2, s_2) \vphantom{\int} \right]\\
		&-hE\left[\iiiint K_h(t_1-s_1)K_h(t_2-s_2) Z(s_1)Z(s_2)^T \right.\\
		& \left. \times E\{Y(t_1)|X(t),Z(s),N(t,s),t,s \in [0, 1] \} \right.\\
		&\left.\times E\{Y(t_2)|X(t),Z(s),N(t,s),t,s \in [0, 1] \} dN(t_1,s_1) dN(t_2, s_2) )\vphantom{\int} \right]\\
		&=h E\left[\iiiint K_h(t_1-s_1)K_h(t_2-s_2) Z(s_1)Z(s_2)^T E\{\epsilon(t_1)\epsilon(t_2) \}dN(t_1,s_1)dN(t_2, s_2)\right]\\
		&=\int K(z)^2 dz \int E\{Z(s)Z(s)^T  \}\eta(s,s)\sigma^2(s)ds + O(h^2)\\
		&:=\Sigma_\gamma + O(h^2),
	\end{align*}
	\begin{align*}
		J_2&=h\var\left[\iint K_h(t-s)Z(s)\{Z(t)-Z(s)\}^T\gamma_0 dN(t,s)\right ]\\
		&=hE\left[\iiiint K_h(t_1-s_1)K_h(t_2-s_2)Z(s_1)\{Z(t_1)-Z(s_1)\}^T\gamma_0\gamma_0^T\{Z(t_2)-Z(s_2)\}\right.\\
		&\left.\times Z(s_2)^T dN(t_1,s_1) dN(t_2,s_2)\vphantom{\int} \right ]\\
		&-h\left(\iint K_h(t-s) E[Z(s)\{Z(t)-Z(s)\}]^T\gamma_0 dN(t,s)\right )^{\otimes2}\\
		&=O(h).
	\end{align*}
	To prove the asymptotic normality, we verify the Lyapunov condition. Note that
	$$(nh)^{1/2}U^f(\beta_0,\gamma_0) = n^{-1/2}\sum_{i=1}^n\psi_i = \sum_{i=1}^n n^{1/2} n^{-1} \psi_i,$$
	then similar to the calculation of variance,
	$$\sum_{i=1}^n E \{\mid n^{1/2} n^{-1} \psi_i - E(n^{1/2} n^{-1} \psi_i) \mid^3\} = nO\{(nh)^{3/2} n^{-3} h^{-2}\} = O\{(nh)^{-1/2}\}.$$
	Thus 
	$$
	\mbox{var}(\psi_i) = \Sigma_{\textcolor{black}{\gamma_0}} + O(h).
	$$
	Consequently,
	$$
	\sqrt{nh}[U^f(\beta_0, \gamma_0) - E\{ U^f(\beta_0, \gamma_0)\}]\overset{d}{\rightarrow} N(0, \Sigma_{\textcolor{black}{\gamma_0}}).
	$$
	As $n^{1/2}h^{5/2} \rightarrow 0,$ we have
	\begin{equation*}
		\sqrt{nh}A_{\gamma_0}(\hat{\gamma} - \gamma)  \overset{d}{\rightarrow} N(0, \Sigma_{\textcolor{black}{\gamma_0}}). 
	\end{equation*}
\end{proof}

\subsection{Proof of Theorem 5}
\begin{theorem}\label{Ch4: cor cao}
	Under conditions (C8), C(10) and (C11),
		let $\hat{\theta}_w = (\hat{\beta}_w^T, \hat{\gamma}_w^T)^T$ and $\theta_0=(\beta_0^T, \gamma_0^T)^T,$ the asymptotic distributions of $\hat\theta_w$ satisfies
		\begin{align*}
			\sqrt{nh}(\hat\theta_w-\theta_0 )&\stackrel{d}{\to}N(0,A_{\textcolor{black}{\theta_0}}^{-1}\Sigma_{\textcolor{black}{\theta_0}} A_{\textcolor{black}{\theta_0}}^{-1}),
		\end{align*}
		where \begin{align*}
			A_{\textcolor{black}{\theta_0}}&=\int\E\{W(t,t)W(t,t)^T\}\eta(t,t)\df t \quad\text{and}\\
			\Sigma_{\textcolor{black}{\theta_0}}&=\int K(z)^2\df z\int\E\{W(t,t)W(t,t)^T \}\sigma^2(t)\eta(t,t)\df t.
		\end{align*}
	\end{theorem}
	\begin{proof}
		$\hat{\theta}_w$ solves the following estimating equation
		\begin{align}
			U_w(\theta) &= n^{-1}\sum_{i=1}^{n}\iint K_h(t-s)W_i(t,s)\{Y_i(t) - W_i(t,s)^T\theta \}dN_i(t,s). \tag{A.5} \label{Ch4: ee theta}
		\end{align}
		
		We first show the consistency of $\hat\theta_w$. Let $s=t+hz$, then
		By the law of large numbers, and let $s=t+hz$, then
		\begin{align*}
			U_w(\theta)&=n^{-1}\sum_{i=1}^{n}\iint K_h(t-s)W_i(t,s)\{Y_i(t) - W_i(t,s)^{T}{\theta} \}dN_i(t,s)\to u_\theta(\theta)
		\end{align*}
		as $n\to\infty$, where
		\begin{align*}
			u_w(\theta)&=\int \E\{W(t,t)W(t,t)^T\}\eta(t,t)d t(\theta_0-\theta).
		\end{align*}
		Under condition (C11), $\theta_0$ is the unique solution to $u_w(\theta)=0$. As $\hat\theta_w$ solves the estimation equation $U_w(\theta)=0$, it follows from the convexity lemma \citep{ag1982} that $\hat\theta_w \to \theta_0$ in probability.
		
		We next show the asymptotic normality of $\hat\theta_w$. Let $\mathcal P_n$ and $\mathcal P$ denote the empirical measure and the true probability measure, respectively. Then we have,
		\begin{align}
			\sqrt{nh}U_w(\theta)&=\sqrt{nh}(\mathcal P_n-\mathcal P)\left[\iint K_h(t-s)W(t,s)\{Y(t) - W(t,s)^T{\theta} \}dN(t,s)\right]\notag \\
			&+\sqrt{nh}\E\left[\iint K_h(t-s)W(t,s)\{Y(t) - W(t,s)^{T}{\theta} \}dN(t,s)\right].\tag{A.6} \label{Ch4: eq cor PP}
		\end{align}
		Consider the class of functions
		\begin{align*}
			\left\{\sqrt{h}\iint K_h(t-s)W(t,s)\{Y(t) - W(t,s)^{T}{\theta} \}dN(t,s):|\theta-\theta_0|<\epsilon \right\}
		\end{align*}
		for a given constant $\epsilon.$ Note that the functions in this class are differentiable and their corresponding first-order derivatives are bounded according to condition (C11), therefore the functions in this class are Lipschitz continuous in $\theta$ and the Lipschitz constant is uniformly bounded by $\|\iint K_h(t-s) W(t,s) W(t,s)^T dN(t,s)\|_{\infty}.$ It can be shown that this class is a P-Donsker class \citep{vw95}. As a result, we obtain that for $|\theta-\theta_0|<M(nh)^{-1/2}$ and some constant $M$, the first term in (\ref{Ch4: eq cor PP}) is equal to
		\begin{align*}
			&\sqrt{nh}(\mathcal P_n-\mathcal P)\left[\iint K_h(t-s)W(t,s)\{Y(t) - W(t,s)^{T}{\theta_0} \}dN(t,s)\right]+o_P(1)\\
			&=\sqrt{nh}[U_w(\theta_0)-\E\{U_w(\theta_0)\} ]+o_P(1).
		\end{align*}
		As to the second term in (\ref{Ch4: eq cor PP}), it's equal to
		
		\begin{align*}
			&\sqrt{nh}\E\left[\iint K_h(t-s)W(t,s)\{Y(t) - W(t,s)^{T}{\theta} \}dN(t,s)\right]\\
			&=\sqrt{nh}\E\Big(\iint K_h(t-s)W(t,s)[W(t,s)^T(\theta_0-\theta)+\{W(t,t)-W(t,s)\}^T\theta_0+\epsilon(t) ]dN(t,s)\Big)\\
			&=\sqrt{nh}\E\left\{\iint K_h(t-s)W(t,s)W(t,s)^TdN(t,s)\right\}(\theta_0-\theta)\\
			&+\sqrt{nh}\E\left[\iint K_h(t-s)W(t,s)\{W(t,t)-W(t,s)\}^TdN(t,s)\right]\theta_0:=I_1+I_2,
		\end{align*}
		where, by letting $s=t+hz$,
		\begin{align*}
			I_1&=\sqrt{nh}\iint K(z)\E\{W(t,t+hz)W(t,t+hz)^T\}\eta(t,t+hz)d td z(\theta_0-\theta)\\
			&=\sqrt{nh}\int\E\{W(t,t)W(t,t)^T\}\eta(t,t)d t(\theta_0-\theta)+O(n^{1/2}h^{5/2})\\
			&:=\sqrt{nh}A_\theta(\theta_0-\theta)+O(n^{1/2}h^{5/2}),\\
			I_2&=\sqrt{nh}\iint K(z)\E[W(t,t+hz)\{W(t,t)-W(t,t+hz)\}^T]\eta(t,t+hz)d td z\theta_0\\
			&=O(n^{1/2}h^{5/2}).
		\end{align*}
		Consequently,
		\begin{align*}
			\sqrt{nh}U_w(\theta)&=\sqrt{nh}[U_w(\theta_0)-\E\{U_w(\theta_0)\} ]+\sqrt{nh}A_{\textcolor{black}{\theta_0}}(\theta_0-\theta)+o(1).
		\end{align*}
		As $\hat\theta_w$ solves (\ref{Ch4: ee theta}),
		\begin{align*}
			\sqrt{nh}[U_w(\theta_0)-\E\{U_w(\theta_0)\}]+o(1)=\sqrt{nh}A_{\textcolor{black}{\theta_0}}(\hat\theta_w-\theta_0).
		\end{align*}
		Let $\psi_i=\sqrt{h}\iint K_h(t-s)W_i(t,s)\{Y_i(t) - W_i(t,s)^{T}{\theta_0} \}dN_i(t,s)$. Then $\psi_i$s are i.i.d. and $\sqrt{nh}U_w(\theta_0)=n^{-1/2}\sum_{i=1}^n\psi_i$.
		\begin{align*}
			\var(\psi)&=\E[\var\{\psi|W(t,s),N(t,s),t,s \in [0,1]\}]+\var[\E\{\psi|W(t,s),N(t,s),t,s \in [0,1]\}]\\
			&=J_1+J_2,
		\end{align*}
		where
		\begin{align*}
			J_1
			&=h\E\left(\iiiint K_h(t_1-s_1)K_h(t_2-s_2)W(t_1,s_1)W(t_2,s_2)^T\E\{\epsilon (t_1)\epsilon (t_2)\}\right.\\
			&\left.\vphantom{\int}\times dN(t_1,s_1)dN(t_2,s_2)\right)\\
			&=\int K(z)^2d z\int\E\{W(t,t)W(t,t)^T \}\sigma^2(t)\eta(t,t)d t+O(h^2),\\
			J_2&=h\var\left[\iint K_h(t-s)W(t,s)\{W(t,t) - W(t,s)\}^{T}{\theta_0}dN(t,s)\right]\\
			&=h\E\left[\iiiint K_h(t_1-s_1)K_h(t_2-s_2)W(t_1,s_1)\{W(t_1,t_1) - W(t_1,s_1)\}^T \theta_0 \theta_0^T\right.\\
			&\left.\vphantom{\int} \times \{W(t_2,t_2) - W(t_2,s_2)\} W(t_2,s_2)^TdN(t_1,s_1)dN(t_2,s_2) \right]\\
			&-h\left(\iint K_h(t-s)\E[W(t,s)\{W(t,t) - W(t,s)\}^{T}]{\theta_0}dN(t,s)\right)^{\otimes2}\\
			&=O(h).
		\end{align*}
		Thus
		\begin{align*}
			\var(\psi)&=\int K(z)^2d z\int\E\{W(t,t)W(t,t)^T \}\sigma^2(t)\eta(t,t)d t+O_P(h^2):=\Sigma_\theta+O(h).
		\end{align*}
		To prove the asymptotic normality, we verify the Lyapunov condition. Note that
		$$(nh)^{1/2}U_w(\theta_0) = n^{-1/2}\sum_{i=1}^n\psi_i = \sum_{i=1}^n n^{1/2} n^{-1} \psi_i,$$
		then similar to the calculation of variance,
		$$\sum_{i=1}^n E \{\mid n^{1/2} n^{-1} \psi_i - E(n^{1/2} n^{-1} \psi_i) \mid^3\} = nO\{(nh)^{3/2} n^{-3} h^{-2}\} = O\{(nh)^{-1/2}\}.$$
		Consequently,
		\begin{align*}
			\sqrt{nh}[U_w(\theta)-\E\{U_w(\theta_0)\} ]\stackrel{d}{\to}N(0,\Sigma_{\textcolor{black}{\theta_0}}).
		\end{align*}
		As $n^{1/2}h^{5/2}\to 0$, we have
		\begin{align*}
			\sqrt{nh}A_{\textcolor{black}{\theta_0}}(\hat\theta_w-\theta_0)\stackrel{d}{\to}N(0,\Sigma_{\textcolor{black}{\theta_0}}).
		\end{align*}
	\end{proof}

	\section{Additional simulation results}
	In this section, we present additional simulation results of gold standard and with different bandwidths. Table \ref{Ch3: table gold} summarizes simulation results for the fully observed longitudinal process. This means that we know the observed values of $X(\cdot), Y(\cdot)$ and $Z(\cdot)$ at any time point $t.$ This is the best results one can possibly get.
	\begin{table}[!ht]
		\caption{1000 simulation results for inference of $\alpha,\gamma$ and $\beta$ with fully observed data} \label{Ch3: table gold}
		\begin{center}\small
			\scalebox{0.75}{\begin{tabular}{llrrrrrrrrrrrrr}
					\hline
					\multirow{2}*{$E\{Z(t)\}$} & \multirow{2}*{$n$}& \multicolumn{4}{c}{$\alpha$} & \multicolumn{4}{c}{$\gamma$} & \multicolumn{4}{c}{$\beta$}\\
					&&	Bias	&	SD	&	SE	&	CP	&	Bias	&	SD	&	SE	&	CP	&	Bias	&	SD	&	SE	&	CP\\
					\hline
					2 & 100 &$0.009$ & $0.214$ & $0.199$ & $93$ & $-0.003$ & $0.089$ & $0.084$ & $93$ & $-0.001$ & $0.083$ & $0.081$ & $94$	\\
					& $400$ & $-0.004$ & $0.102$ & $0.103$ & $95$ & $0.001$ & $0.043$ & $0.043$ & $95$ & $0.001$ & $0.042$ & $0.041$ & $93$ \\ 
					& $900$ & $-0.001$ & $0.068$ & $0.069$ & $96$ & $-0.000$ & $0.029$ & $0.029$ & $95$ & $0.000$ & $0.028$ & $0.028$ & $95$ \\ 
					\hline
					$0.5+t$	& $100$ & $0.001$ & $0.134$ & $0.132$ & $94$ & $0.001$ & $0.081$ & $0.080$ & $94$ & $-0.001$ & $0.084$ & $0.081$ & $94$	\\
					& $400$ & $-0.000$ & $0.066$ & $0.067$ & $95$ & $0.002$ & $0.041$ & $0.041$ & $94$ & $0.000$ & $0.042$ & $0.042$ & $94$\\
					& $900$ & $-0.002$ & $0.046$ & $0.045$ & $94$ & $0.001$ & $0.027$ & $0.027$ & $95$ & $0.001$ & $0.027$ & $0.028$ & $95$ \\
					\hline
					$0.5+t^2$	& $100$ & $-0.002$ & $0.134$ & $0.125$ & $92$ & $0.004$ & $0.083$ & $0.079$ & $93$ & $-0.003$ & $0.085$ & $0.081$ & $93$ 	\\
					&$400$ & $0.002$ & $0.063$ & $0.064$ & $95$ & $-0.001$ & $0.041$ & $0.040$ & $94$ & $0.001$ & $0.041$ & $0.041$ & $94$ \\ 
					&$900$ & $0.0002$ & $0.044$ & $0.043$ & $95$ & $0.0004$ & $0.027$ & $0.027$ & $94$ & $-0.0003$ & $0.028$ & $0.028$ & $95$ \\  
					\hline
					$2\sin(2\pi t)$ & $100$ & $0.003$ & $0.113$ & $0.110$ & $94$ & $0.000$ & $0.037$ & $0.037$ & $94$ & $0.003$ & $0.085$ & $0.082$ & $94$ \\
					& $400$ & $-0.000$ & $0.056$ & $0.056$ & $94$ & $-0.001$ & $0.019$ & $0.019$ & $94$ & $-0.000$ & $0.041$ & $0.042$ & $94$ \\
					& $900$ & $-0.001$ & $0.038$ & $0.037$ & $95$ & $0.000$ & $0.012$ & $0.013$ & $95$ & $0.000$ & $0.028$ & $0.028$ & $96$ \\ 
					\hline
			\end{tabular}}
		\end{center}
		\footnotesize{Note: ``Bias" is the empirical bias, ``SD'' is the sample standard deviation, ``SE'' is the average of the standard error estimates, ``CP''$/100$ represents the coverage probability of the $95\%$ confidence interval for $\hat{\beta}$.}
	\end{table}
	
	We summarize simulation results with bandwidth $h=n^{-0.7}$ and $h=n^{-0.8}$ in Table \ref{Ch3: table beta h=n^-.7} and Table \ref{Ch3: table beta h=n^-.8} respectively. We can see that different choices of bandwidths do not affect inference of $\hat{\beta}_p$ and $\hat{\beta}_c$ much. 
	\begin{table}[!ht]
		\caption{1000 simulation results for inference of $\beta$ with $h=n^{-0.7}$}\label{Ch3: table beta h=n^-.7}
		\begin{center}\small
			\scalebox{0.8}{\begin{tabular}{lrrrrrrrrrrrr}
					\hline
					& \multicolumn{4}{c}{Na\"ive} & \multicolumn{4}{c}{PLM} & \multicolumn{4}{c}{Centering}\\
					&	Bias	&	SD	&	SE	&	CP	&	Bias	&	SD	&	SE	&	CP	&	Bias	&	SD	&	SE	&	CP\\
					\hline
					\multicolumn{3}{l}{ Independent covariates} \\
					\multicolumn{3}{l}{ $E\{Z(t)\}=2$} \\
					$100$ & $0.002$ & $0.116$ & $0.114$ & $93$ & $0.002$ & $0.122$ & $0.117$ & $94$ & $0.002$ & $0.122$ & $0.117$ & $94$ \\ 
					$400$ & $0.002$ & $0.055$ & $0.058$ & $95$ & $0.002$ & $0.058$ & $0.060$ & $96$ & $0.002$ & $0.058$ & $0.060$ & $95$ \\ 
					$900$ & $-0.0002$ & $0.040$ & $0.039$ & $94$ & $-0.0004$ & $0.041$ & $0.040$ & $94$ & $-0.0004$ & $0.041$ & $0.040$ & $94$ \\ 
					\multicolumn{3}{l}{ $E\{Z(t)\}=0.5+t$} \\
					$100$ & $-0.064$ & $0.117$ & $0.114$ & $89$ & $0.001$ & $0.121$ & $0.116$ & $93$ & $0.001$ & $0.121$ & $0.116$ & $93$ \\ 
					$400$ & $-0.061$ & $0.061$ & $0.058$ & $81$ & $0.003$ & $0.063$ & $0.060$ & $94$ & $0.003$ & $0.063$ & $0.060$ & $94$ \\ 
					$900$ & $-0.063$ & $0.040$ & $0.039$ & $62$ & $0.0001$ & $0.042$ & $0.040$ & $94$ & $0.00004$ & $0.042$ & $0.040$ & $94$ \\ 
					\multicolumn{3}{l}{ $E\{Z(t)\}=0.5+t^2$} \\
					$100$ & $-0.059$ & $0.117$ & $0.114$ & $91$ & $0.003$ & $0.124$ & $0.116$ & $92$ & $0.003$ & $0.124$ & $0.116$ & $92$ \\ 
					$400$ & $-0.058$ & $0.060$ & $0.058$ & $83$ & $0.002$ & $0.062$ & $0.060$ & $94$ & $0.002$ & $0.062$ & $0.060$ & $94$ \\ 
					$900$ & $-0.059$ & $0.040$ & $0.039$ & $66$ & $0.001$ & $0.040$ & $0.040$ & $95$ & $0.001$ & $0.040$ & $0.040$ & $95$ \\ 
					\multicolumn{3}{l}{ $E\{Z(t)\}=2\sin(2\pi t)$} \\
					$100$ & $0.234$ & $0.136$ & $0.132$ & $57$ & $0.002$ & $0.122$ & $0.116$ & $94$ & $0.002$ & $0.122$ & $0.116$ & $94$ \\ 
					$400$ & $0.231$ & $0.067$ & $0.068$ & $6$ & $0.002$ & $0.060$ & $0.059$ & $95$ & $0.002$ & $0.060$ & $0.059$ & $95$ \\ 
					$900$ & $0.230$ & $0.046$ & $0.045$ & $0$ & $0.002$ & $0.040$ & $0.040$ & $95$ & $0.002$ & $0.040$ & $0.040$ & $95$ \\ 
					\hline
					\multicolumn{3}{l}{ Uncorrelated covariates} \\
					\multicolumn{3}{l}{ $E\{Z(t)\}=2$} \\
					$100$ & $-0.007$ & $0.219$ & $0.190$ & $90$ & $-0.008$ & $0.231$ & $0.195$ & $89$ & $-0.008$ & $0.231$ & $0.196$ & $89$ \\ 
					$400$ & $0.005$ & $0.111$ & $0.103$ & $92$ & $0.005$ & $0.117$ & $0.107$ & $92$ & $0.005$ & $0.117$ & $0.107$ & $92$ \\ 
					$900$ & $0.002$ & $0.071$ & $0.070$ & $95$ & $0.002$ & $0.074$ & $0.073$ & $94$ & $0.002$ & $0.074$ & $0.073$ & $94$ \\ 
					\multicolumn{3}{l}{ $E\{Z(t)\}=0.5+t$} \\
					$100$ & $-0.066$ & $0.206$ & $0.190$ & $90$ & $0.001$ & $0.217$ & $0.194$ & $91$ & $0.001$ & $0.217$ & $0.195$ & $91$ \\ 
					$400$ & $-0.065$ & $0.106$ & $0.104$ & $89$ & $-0.001$ & $0.112$ & $0.107$ & $94$ & $-0.001$ & $0.112$ & $0.107$ & $94$ \\ 
					$900$ & $-0.064$ & $0.070$ & $0.070$ & $85$ & $-0.001$ & $0.073$ & $0.073$ & $95$ & $-0.001$ & $0.073$ & $0.073$ & $95$ \\ 
					\multicolumn{3}{l}{ $E\{Z(t)\}=0.5+t^2$} \\
					$100$ & $-0.063$ & $0.205$ & $0.192$ & $91$ & $0.002$ & $0.215$ & $0.197$ & $90$ & $0.002$ & $0.214$ & $0.197$ & $91$ \\ 
					$400$ & $-0.066$ & $0.110$ & $0.104$ & $88$ & $-0.006$ & $0.116$ & $0.107$ & $92$ & $-0.006$ & $0.116$ & $0.107$ & $92$ \\ 
					$900$ & $-0.059$ & $0.072$ & $0.070$ & $85$ & $0.002$ & $0.076$ & $0.073$ & $94$ & $0.002$ & $0.076$ & $0.073$ & $94$ \\ 
					\multicolumn{3}{l}{ $E\{Z(t)\}=2\sin(2\pi t)$} \\
					$100$ & $0.245$ & $0.218$ & $0.206$ & $74$ & $0.003$ & $0.218$ & $0.194$ & $90$ & $0.003$ & $0.218$ & $0.194$ & $91$ \\ 
					$400$ & $0.227$ & $0.113$ & $0.111$ & $46$ & $-0.005$ & $0.110$ & $0.107$ & $94$ & $-0.005$ & $0.110$ & $0.107$ & $94$ \\ 
					$900$ & $0.232$ & $0.075$ & $0.075$ & $13$ & $0.002$ & $0.075$ & $0.072$ & $93$ & $0.002$ & $0.075$ & $0.072$ & $93$ \\ 
					\hline
			\end{tabular}}
		\end{center}
		\footnotesize{Note: ``Bias" is the empirical bias, ``SD'' is the sample standard deviation, ``SE'' is the average of the standard error estimates, ``CP''$/100$ represents the coverage probability of the $95\%$ confidence interval for $\hat{\beta}$.}
	\end{table}
	
	\begin{table}[!ht]
		\caption{1000 simulation results for inference of $\beta$ with $h=n^{-0.8}$}\label{Ch3: table beta h=n^-.8}
		\begin{center}\small
			\scalebox{0.8}{\begin{tabular}{lrrrrrrrrrrrr}
					\hline
					& \multicolumn{4}{c}{Na\"ive} & \multicolumn{4}{c}{PLM} & \multicolumn{4}{c}{Centering}\\
					&	Bias	&	SD	&	SE	&	CP	&	Bias	&	SD	&	SE	&	CP	&	Bias	&	SD	&	SE	&	CP\\
					\hline
					\multicolumn{3}{l}{ Independent covariates} \\
					\multicolumn{3}{l}{ $E\{Z(t)\}=2$} \\
					$100$ & $0.005$ & $0.109$ & $0.113$ & $95$ & $0.005$ & $0.114$ & $0.114$ & $94$ & $0.005$ & $0.114$ & $0.114$ & $94$ \\ 
					$400$ & $0.002$ & $0.056$ & $0.058$ & $95$ & $0.003$ & $0.059$ & $0.059$ & $95$ & $0.003$ & $0.059$ & $0.059$ & $95$ \\ 
					$900$ & $-0.001$ & $0.038$ & $0.039$ & $95$ & $-0.001$ & $0.040$ & $0.040$ & $94$ & $-0.001$ & $0.040$ & $0.040$ & $94$ \\ 
					\multicolumn{3}{l}{ $E\{Z(t)\}=0.5+t$} \\
					$100$ & $-0.067$ & $0.120$ & $0.114$ & $88$ & $-0.003$ & $0.125$ & $0.114$ & $92$ & $-0.003$ & $0.125$ & $0.114$ & $92$ \\ 
					$400$ & $-0.063$ & $0.061$ & $0.058$ & $80$ & $0.001$ & $0.064$ & $0.059$ & $91$ & $0.001$ & $0.064$ & $0.059$ & $92$ \\ 
					$900$ & $-0.064$ & $0.041$ & $0.039$ & $63$ & $-0.001$ & $0.043$ & $0.040$ & $94$ & $-0.001$ & $0.043$ & $0.040$ & $94$ \\ 
					\multicolumn{3}{l}{ $E\{Z(t)\}=0.5+t^2$} \\
					$100$ & $-0.062$ & $0.116$ & $0.114$ & $90$ & $-0.0001$ & $0.119$ & $0.115$ & $93$ & $-0.0002$ & $0.119$ & $0.115$ & $93$ \\ 
					$400$ & $-0.061$ & $0.060$ & $0.058$ & $81$ & $-0.001$ & $0.061$ & $0.059$ & $94$ & $-0.001$ & $0.061$ & $0.059$ & $94$ \\ 
					$900$ & $-0.058$ & $0.040$ & $0.039$ & $69$ & $0.003$ & $0.042$ & $0.040$ & $94$ & $0.003$ & $0.042$ & $0.040$ & $94$ \\ 
					\multicolumn{3}{l}{ $E\{Z(t)\}=2\sin(2\pi t)$} \\
					$100$ & $0.232$ & $0.137$ & $0.133$ & $58$ & $0.001$ & $0.124$ & $0.114$ & $93$ & $0.001$ & $0.123$ & $0.115$ & $93$ \\ 
					$400$ & $0.227$ & $0.068$ & $0.068$ & $8$ & $-0.001$ & $0.061$ & $0.059$ & $93$ & $-0.001$ & $0.061$ & $0.059$ & $93$ \\ 
					$900$ & $0.228$ & $0.046$ & $0.045$ & $0$ & $-0.0001$ & $0.040$ & $0.040$ & $94$ & $-0.0001$ & $0.040$ & $0.040$ & $94$ \\ 
					\hline
					\multicolumn{3}{l}{ Uncorrelated covariates} \\
					\multicolumn{3}{l}{ $E\{Z(t)\}=2$} \\
					$100$ & $-0.002$ & $0.211$ & $0.189$ & $91$ & $-0.001$ & $0.225$ & $0.192$ & $90$ & $-0.001$ & $0.225$ & $0.192$ & $90$ \\ 
					$400$ & $-0.004$ & $0.105$ & $0.103$ & $93$ & $-0.004$ & $0.111$ & $0.106$ & $92$ & $-0.004$ & $0.111$ & $0.106$ & $92$ \\ 
					$900$ & $-0.004$ & $0.072$ & $0.070$ & $94$ & $-0.005$ & $0.075$ & $0.072$ & $94$ & $-0.005$ & $0.075$ & $0.072$ & $94$ \\ 
					\multicolumn{3}{l}{ $E\{Z(t)\}=0.5+t$} \\
					$100$ & $-0.059$ & $0.207$ & $0.186$ & $90$ & $0.008$ & $0.219$ & $0.188$ & $88$ & $0.008$ & $0.218$ & $0.188$ & $88$ \\ 
					$400$ & $-0.065$ & $0.108$ & $0.104$ & $88$ & $0.0001$ & $0.112$ & $0.106$ & $92$ & $0.0001$ & $0.112$ & $0.106$ & $92$ \\ 
					$900$ & $-0.065$ & $0.070$ & $0.071$ & $85$ & $-0.001$ & $0.073$ & $0.072$ & $95$ & $-0.001$ & $0.073$ & $0.072$ & $95$ \\ 
					\multicolumn{3}{l}{ $E\{Z(t)\}=0.5+t^2$} \\
					$100$ & $-0.057$ & $0.203$ & $0.190$ & $90$ & $0.006$ & $0.213$ & $0.191$ & $90$ & $0.006$ & $0.213$ & $0.192$ & $90$ \\ 
					$400$ & $-0.063$ & $0.108$ & $0.104$ & $88$ & $-0.002$ & $0.114$ & $0.106$ & $93$ & $-0.002$ & $0.114$ & $0.106$ & $93$ \\ 
					$900$ & $-0.059$ & $0.071$ & $0.071$ & $86$ & $0.001$ & $0.075$ & $0.072$ & $94$ & $0.001$ & $0.075$ & $0.072$ & $94$ \\ 
					\multicolumn{3}{l}{ $E\{Z(t)\}=2\sin(2\pi t)$} \\
					$100$ & $0.239$ & $0.229$ & $0.208$ & $74$ & $-0.001$ & $0.216$ & $0.195$ & $90$ & $-0.001$ & $0.216$ & $0.195$ & $90$ \\ 
					$400$ & $0.231$ & $0.118$ & $0.112$ & $46$ & $-0.003$ & $0.116$ & $0.106$ & $92$ & $-0.003$ & $0.116$ & $0.106$ & $92$ \\ 
					$900$ & $0.234$ & $0.080$ & $0.076$ & $15$ & $0.003$ & $0.077$ & $0.072$ & $94$ & $0.003$ & $0.077$ & $0.072$ & $94$ \\ 
					\hline
			\end{tabular}}
		\end{center}
		\footnotesize{Note: ``Bias" is the empirical bias, ``SD'' is the sample standard deviation, ``SE'' is the average of the standard error estimates, ``CP''$/100$ represents the coverage probability of the $95\%$ confidence interval for $\hat{\beta}$.}
	\end{table}


\begin{thebibliography}{99}
		\bibitem[Andersen and Gill(1982)]{ag1982}Andersen, P. K., Gill, R. D. (1982). Cox's regression model for counting processes: a large sample study. {\em The Annals of Statistics}, {10}, 1100--1120.
		
	\bibitem[Bates et al.(2022)]{bates2022} Bates, S., Kennedy, E., Tibshirani, R., Ventura, V., Wasserman, L. (2022). Causal inference with orthogonalized regression adjustment: taming the phantom. {\em arXiv:2201.13451v2}.
	
	\bibitem[Bekris et al.(2010)]{bekris2010} Bekris, L. M., Yu, C.-E., Bird, T. D., Tsuang, D. W. (2010). Genetics of Alzheimer Disease. {\em Journal of Geriatric Psychiatry and Neurology}, {23}, 213--227.
	
	\bibitem[Cao et al.(2016)]{cao16}Cao, H., Li, J., Fine, J. P. (2016). On last observation carried forward and asynchronous longitudinal regression analysis. {\em Electronic Journal of Statistics}, {10}, 1155--1180.
	
	\bibitem[Cao et al.(2015)]{cao15} Cao, H., Zeng, D., Fine, J. P. (2015). Regression analysis of sparse asynchronous longitudinal data. {\em Journal of the Royal Statistical Society: Series B}, {77}, 755--776.
	
	\bibitem[Chen and Cao(2017)]{chen17}Chen, L., Cao, H. (2017). Analysis of asynchronous longitudinal data with partially linear models. {\em Electronic Journal of Statistics}, {11},  1549--1569.
	
	\bibitem[Ding(2021)]{ding2021} Ding, P. (2021). The Frisch-Waugh-Lovell theorem for standard errors. {\em Statistics and Probability Letters}, {168}, 108945.  
	
	\bibitem[Fan(1992)]{fan92}Fan, J. (1992). Design-adaptive nonparametric regression. {\em Journal of the American statistical Association}, {87}, 998-1004.
	
	\bibitem[Fan and Li(2004)]{fan2004}Fan, J., Li, R. (2004). New estimation and model selection procedures for semiparametric modeling in longitudinal data analysis. {\em Journal of the American Statistical Association}, {99}, 710--723.
	
	
	\bibitem[Frisch and Waugh(1933)]{frisch1933} Frisch, R., Waugh, F. V. (1933). Partial time regressions as compared with individual trends. {\em Econometrica}, {1}, 387--401.
	
	\bibitem[Kristensen et al.(2019)]{kristensen2019}Kristensen T. D., Mandl R. C. W., Raghava J. M., Jessen K., Jepsen J. R. M., Fagerlund B., Glenth{\o}j L. B., Wenneberg C., Krakauer K., Pantelis C., Nordentoft M., Glenth{\o}j B. Y., Ebdrup B. H. (2019). Widespread higher fractional anisotropy associates to better cognitive functions in individuals at ultra-high risk for psychosis. {\em Human Brain Mapping}, {40}, 5185--5201.
	
	\bibitem[Li et al.(2022)]{lizhu2022}Li, T., Li, T., Zhu, Z., Zhu, H. (2022). Regression analysis of asynchronous longitudinal functional and scalar data. {\em Journal of the American Statistical Association}, in press. 
	
	\bibitem[Liang and Zeger(1986)]{liangzeger86}Liang, K.-Y., Zeger, S. L. (1986). Longitudinal data analysis using generalized linear models. {\em Biometrika}, {73}, 13--22.
	
	\bibitem[Lin and Ying(2001)]{lin01}Lin, D. Y., Ying, Z. (2001). Semiparametric and nonparametric regression analysis of longitudinal data. {\em Journal of the American Statistical Association}, {96}, 103--126.
	
	\bibitem[Little and Rubin(2014)]{little2014} Little, R. J. A., Rubin, D. B. (2014). {\em Statistical analysis with missing data.} John Wiley \& Sons.
	
	\bibitem[Lovell(1963)]{lovell1963}Lovell, M. C. (1963). Seasonal adjustment of economic time series and multiple regression analysis. {\em Journal of the American Statistical Association}, {58}, 993--1010.
	
	\bibitem[Lovell(2008)]{lovell2008}Lovell, M. C. (2008). A simple proof of the FWL theorem. {\em Journal of Economic Education}, {39}, 88--91.
	
	\bibitem[Nadaraya(1964)]{nadaraya64}Nadaraya, E. A. (1964). On estimating regression. {\em Theory of Probability \& Its Applications}, {9}, 141-142.
	
	\bibitem[Pepe and Anderson(1994)]{pepe1994} Pepe, M. S., Anderson, G. L. (1994). A cautionary note on inference for marginal regression models with longitudinal data and general correlated response data. {\em Communications in Statistics - Simulation and Computation}, {23}, 939--951.
	
	\bibitem [Qian and Wang(2017)]{qian11}Qian, L., Wang, S. (2017) Subject-wise empirical likelihood inference in partial linear models for longitudinal data. {\em Computational Statistics and Data Analysis}, {111}, 77--87.
	
	\bibitem[van der Vaart and Wellner(1996)]{vw95}van der Vaart A. W., Wellner, J. A. (1996). {\em Weak convergence and empirical processes.} {Springer, New York.}
	
	\bibitem[Watson(1964)]{watson64}Watson, G. S. (1964). Smooth regression analysis. {\em Sankhy\=a: The Indian Journal of Statistics, Series A}, {26}, 359--372.
	
	\bibitem[Sent\"{u}rk et al. (2013)]{senturk13}Sent\"{u}rk, D., Dalrymple, L. S., Mohammed, S. M., Kaysen, G. A., Nguyen, D. V. (2013). Modeling time-varying effects with generalized and unsynchronized longitudinal data.
	{\em Statistics in Medicine}, {32}, 2971--2987.
	
	\bibitem[Sun et al.(2021)]{sun2021} Sun, D., Zhao, H., Sun, J. (2021). Regression analysis of asynchronous longitudinal data with informative observation processes. {\em Computational Statistics and Data Analysis}, 107161.
	
	\bibitem[Xiong and Dubin(2010)]{xiong10}Xiong, X., Dubin, J. A. (2010). A binning method for analyzing mixed longitudinal data measured at distinct time points. {\em Statistics in Medicine}, {29}, 1919--1931.
\end{thebibliography}
\end{document}